\documentclass[12pt]{article}

\usepackage{amsmath}
\counterwithin{equation}{section}
\usepackage[left=1in, right=1in, top=1in, bottom=1in]{geometry}
\usepackage{multirow}
\usepackage{cases}
\usepackage{nicefrac}
\usepackage{caption}
\usepackage{subcaption}
\usepackage{multicol}
\usepackage{amssymb}
\usepackage{amsfonts}
\usepackage{array}
\usepackage{enumitem}
\usepackage{amsthm}
\usepackage{graphicx}
\usepackage{amstext}
\usepackage{mathrsfs}
\usepackage{float}
\usepackage{bbm}
\usepackage{subfig}
\usepackage{soul}
\usepackage{authblk}
\usepackage{fancyhdr}

\theoremstyle{definition}
\newtheorem{theorem}{Theorem}[section] 
\newtheorem{example}{Example}
\newtheorem{definition}{Definition}
\newtheorem{proposition}[theorem]{Proposition}
\newtheorem{corollary}[theorem]{Corollary}
\newtheorem{lemma}[theorem]{Lemma}

\newtheorem{remark}{Remark}

\newcommand{\A}{\mathcal{A}}
\newcommand{\Dd}{\mathbb{D}}
\newcommand{\D}{\mathcal{D}}

\newcommand{\G}{\mathscr{G}}
\renewcommand{\H}{\mathbb{H}}

\newcommand{\Z}{\mathbb{Z}}
\newcommand{\C}{\mathbb{C}}
\newcommand{\R}{\mathbb{R}}
\newcommand{\Q}{\mathbb{Q}}

\newcommand{\N}{\mathbb{N}}

\newcommand{\X}{\mathcal{X}}
\renewcommand{\L}{\mathcal{L}}

\renewcommand{\i}{\hat{\imath}}
\renewcommand{\j}{\hat{\jmath}}
\renewcommand{\k}{\hat{k}}

\newcommand{\lex}{<_{\mathrm{lex}}}
\newcommand{\lexeq}{\leq_{\mathrm{lex}}}
\renewcommand{\a}{\mathbbm{a}}
\renewcommand{\b}{\mathbbm{b}}
\renewcommand{\c}{\mathbbm{c}}
\renewcommand{\d}{\mathbbm{d}}

\everymath{\displaystyle}
\allowdisplaybreaks

\makeatletter

\title{Rotational beta expansions and Schmidt games}

\date{}

\begin{document}

\author[1]{Jonathan Caalim\thanks{jcaalim@math.upd.edu.ph}}
\author[2]{Hajime Kaneko\thanks{kanekoha@math.tsukuba.ac.jp}}
\author[1]{Nathaniel Nollen\thanks{nnollen@math.upd.edu.ph}}

\affil[1]{\small{University of the Philippines - Diliman, Quezon City, Philippines 1101}}
\affil[2]{\small{University of Tsukuba, 1-1-1 Tennodai Tsukuba, Ibaraki 305-8577, Japan}}
\maketitle

\fancyhf{}
\renewcommand\headrulewidth{0pt}
\fancyhead[C]{%
\ifodd\value{page}
  \small\scshape{Jonatahn Caalim, Hajime Kaneko and Nathaniel Nollen}
\else
  \small\scshape{Rotational beta expansions and Schmidt games}
\fi
}
\pagestyle{fancy}

\begin{abstract}
    We consider rotational beta expansions in dimensions 1, 2 and 4 and view them as expansions on real numbers, complex numbers, and quaternions, respectively.
    We give sufficient conditions on the parameters $\alpha, \beta \in (0,1)$ so that particular cylinder sets arising from the expansions are winning or losing Schmidt $(\alpha,\beta)$-game.\\
    
\end{abstract}

\section{Introduction}\label{sec1}

Let $\alpha,\beta\in(0,1)$. 
The Schmidt $(\alpha, \beta)$-game on a complete metric space $(X, \lambda)$ is a game with two players, Alice and Bob, and the following recursive rules \cite{schmidtw}: 

    \begin{enumerate}
        \item At the start, Bob chooses an initial radius $\rho=\rho_0>0$ and an initial center $x_0\in X$. Set $B_0:=\overline{B(x_0,\rho)}$. 
        Meanwhile, Alice chooses a \textit{target set} $S\subseteq X$.

        \item Bob and Alice play alternately. For $n\in \mathbb{N}$, set $\rho_n:=(\alpha\beta)^n\rho$ and for
        \begin{enumerate}
        \item Alice's $n$th turn: Alice chooses $y_n \in X$ such that 
        \[\lambda(x_{n-1},y_n)+\alpha\rho_{n-1}\leq \rho_{n-1},\]
        \item Bob's $n$th turn: Bob chooses $x_n \in X$ such that 
        \[\lambda(y_n,x_n)+\rho_n\leq \alpha \rho_{n-1}.\]
        \end{enumerate}
    \end{enumerate}
Let $B_n := \overline{B(x_n,\rho_n)}$ and $A_n:= \overline{B(y_n,\alpha \rho_{n-1})}$. 
     Then we have 
     \[B_0\supseteq A_1\supseteq B_1\supseteq A_2\supseteq \cdots.\]
     The \textit{outcome} of the game is the unique point $\omega$ such that 
     \[\{\omega\} = \left(\bigcap_{n=1}^\infty A_n\right)\cap \left(\bigcap_{n=0}^\infty B_n\right).\]
     Alice wins the game if $\omega\in S$. 
     Otherwise, Bob wins.

     Given $\alpha,\beta$ and $\rho$, the set $S$ is called \textit{$(\alpha,\beta,\rho)$-winning} if Alice has a strategy to win the game regardless of Bob's play. 
     If $S$ is \textit{$(\alpha,\beta,\rho)$-winning} for all $\rho>0$, we say that $S$ is ($\alpha,\beta$)-winning.
     If $S$ is $(\alpha,\beta)$-winning for any $\beta$, we say that $S$ is \textit{$\alpha$-winning}. 
     A set $S$ is \textit{winning} if $S$ is $\alpha$-winning for some $\alpha$.

      From \cite{schmidtw}, we know that if $S$ is winning, then $S$ is dense. Furthermore, if $X=\R^m$, then the Hausdorff dimension of $S$ is $m$. 
      Winning sets also exhibit large intersection properties.
      Indeed, let $\alpha$ be a real number with $0<\alpha<1$ and let $S_n$ $(n=1,2,\ldots)$ be $\alpha$-winning sets. Then the intersection $\cap_{n\geq 1} S_n$ is also $\alpha$-winning. 
      Lemma 7.13 of \cite{Bugeaud_2012} provides a more flexible result as follows. Let $P_j=\{m_j+d_j n\mid n=0,1,\ldots\}$ ($j=1,2,\ldots$) be an arithmetic progression of positive integers where $m_j$ is the first term and $d_j$ is the common difference. 
      Assume that $\cup_{j\geq 1}P_j$ is a disjoint union partitioning $\mathbb{N}$. Let $\alpha,\beta, \rho$ be positive real numbers with $\alpha<1/3$ and $\beta<1$. Set $\beta_j:=\beta(\alpha\beta)^{-1+d_j}$ and $\rho_j:=\rho(\alpha\beta)^{-1+m_j}$.
      If $S_j$  is an ($\alpha, \beta_j, \rho_j$)-winning set for each $j\in \mathbb{N}$, then $\cap_{j\geq 1}S_j$ is an ($\alpha,\beta,\rho$)-winning set. In particular, the intersection $\cap_{j\geq 1}S_j$ is uncountable. This gives one of the motivations to search for ($\alpha,\beta,\rho$)-winning sets.
      
      Schmidt game is a useful tool to study the ``denseness" of cylinder sets related to numeration systems. In the literature, there are several studies on the $b$-ary expansions $(1<b\in\N)$ and Schmidt games. 
     In 1982, Freiling \cite{FREILING1982226} showed that the set of real numbers containing the digit 0 or 5 in base 6 is $(1/2,1/2)$-winning but $(1/2,1/3)$-\textit{losing}.
     In 2002, Dremov \cite{Dremov} proved that the set of real numbers containing the digit 0 or 3 in base 4 is $(1/2,1/2)$-losing.
     In 2024, Neckrasov and Zhan \cite{neckrasov2024nontrivial} showed that the set of real numbers where $0$ appears at least half of the time in the $n$-tails ($n\in\N$) of its binary expansion is $(\alpha,\beta)$-losing for $\alpha>\beta$.
     Given $b\in\N>1$, Zanger-Tishler and Kalia \cite{ZangerTishler2013OnTW} gave a sufficient condition on $(\alpha,\beta)$ so that the set of real numbers with $0$ or $b-1$ in their $b$-ary expansion is either $(\alpha,\beta)$-winning or $(\alpha,\beta)$-losing.

     In this article, we apply Schmidt games to  rotational beta expansions \cite{akiyama2017rotational}. 
     Let $m \in \mathbb{N}$. Let $\beta>1$ and let $M$ be an element of the special orthogonal group $SO(m)$.
    Let $\L$ be a lattice on $\R^m$ with fundamental domain $\mathcal{X}$.
    Define the rotational beta transformation $\mathbb{T}=\mathbb{T}_{\beta M,\L,\mathcal{X}}: \mathcal{X}\to\mathcal{X}$ by 
    \begin{align}\label{general_beta_rotation}
        \mathbb{T}(z) = \beta M z -d(z),
    \end{align}where $d(z)$ is the unique element of $\L$ such that $\beta M z-d(z)\in\mathcal{X}$.
   The rotational beta expansion of $z\in\mathcal{X}$ is the sequence $\mathbbm{d}(z):=d_1d_2\dots$ where
    the $n$th digit is $d_n=d_n(z):=d(\mathbb{T}^{n-1}(z))$.
    We have \[z=(\beta M)^{-1}d_1+(\beta M)^{-2}d_2+\cdots.\]
    
    The digit set of the rotational beta expansion is 
    $\mathcal{D}:=\left\{d_1(z) \in\L: z\in\mathcal{X}\right\}$.
    An ordered $n$-tuple $d_1d_2\cdots d_n\in\mathcal{D}^n$ is called an admissible $n$-block (or, simply, block) if it is a finite subsequence of $\mathbbm{d}(z)$ for some $z\in\mathcal{X}$.
    For an admissible block $\Omega$, we define the set
    \[\mathcal{C}[\Omega]= \mathcal{C}_{(\beta M,\L,\mathcal{X})}[\Omega]:=\{z\in\mathcal{X}\mid \Omega\text{ is a subblock of } \mathbbm{d}(z)\}.\]
    
    In this paper, we study the winnability of $C[\Omega]$ with respect to Schmidt games for rotational beta expansions in dimensions $m=1,2$ and $4$. 
    The rotational beta expansion may be interpreted as: the classical beta expansion \cite{Renyi1957} if $m=1$;  a numeration system on $\mathbb{C}$ with a complex number base if $m=2$; and an expansion over the real quaternions with respect to a quaternion radix \cite{ojmpreprint}. 
    In Section 2, we look at the winnability of $C[\Omega]$ for $m=1$. Note that we can show $\alpha$-winnability (see Theorem \ref{dwinning}) under certain assumptions, which may be considered reasonable. For comparison, see  \cite[Theorem 1.7]{langeveld2023intermediate}, where Langeveld and Samuel investigated the $\alpha$-winnability of the set badly approximable numbers $x$ (i.e. the orbit of $x$ under the (generalized) beta transformation avoids a particular given point) with respect to (generalized) beta expansions.
    In Section 3, we give sufficient conditions for winnability of $C[\Omega]$ where $m=2$. This provides the first result of the winnability of rotational beta expansion. In Section 4, we treat the case where $m=4$.

\section{Real expansions}

R\'enyi \cite{Renyi1957} introduced the so-called beta expansion on real numbers with real base $b>1$.
This generalizes the $b$-ary expansions on real numbers where the radix $b$ is a natural number.
Given a fixed $b \in \mathbb{R}$ with $b>1$, the $j$th digit $d_j=d_j(x)$ of a real number $x \in [0,1)$ is the unique integer $d\in\{0,1,\dots,s_b\}=:\D$ where 
\[s_b=\begin{cases}
    b-1,&\text{ if }b\in\N\\
    \lfloor b\rfloor,&\text{ if }b\notin\N
\end{cases}\]
such that 
\[b^j\left(x-\sum_{k=1}^{j}d_kb^{-k}\right)\in [0,1).\]
We call the sequence $\mathbbm{d}(x):=d_1d_2\cdots$ the $b$-expansion of $x$.

Let $n\in\N$. 
We identify an $n$-block $a_1a_2\cdots a_n\in\D^n$ with the sequence $a_1a_2\cdots a_n000\cdots \in\D^\N$.
For $\mathbbm{a}=a_1a_2\dots\in\D^\N$, we define
\[\mathbbm{a}(b):=\frac{a_1}{b}+\frac{a_2}{b^2}+\cdots.\]
Given $\mathbbm{a}=a_1a_2\cdots, \mathbbm{b}=b_1b_2\cdots\in\D^\N$, we say that 
$\mathbbm{a}$ is lexicographically less than $\mathbbm{b}$ and write 
\[\mathbbm{a}\lex \mathbbm{b}\]
if there exists $k\in\N$ such that $a_j=b_j$ when $j<k$ and $a_k<b_k$.
We write $\mathbbm{a} \lexeq \mathbbm{b}$ if $\mathbbm{a}\lex \mathbbm{b}$ or $\mathbbm{a}= \mathbbm{b}$.
It is known that if $\mathbbm{a},\mathbbm{b}\in\D^\N$ are admissible, then $\mathbbm{a}\lex \mathbbm{b}$ if and only if $\mathbbm{a}(b)<\mathbbm{b}(b)$. 
Denote the set of admissible $n$-blocks in $\D^n$ by $\D_A^n$.
We can arrange the elements of $\D^n$ (or $\D^n_A$) in an increasing manner with respect to the lexicographic order:
\[\mathbbm{e}_1\lex \mathbbm{e}_2\lex\cdots \lex\mathbbm{e}_N.\] 
We say that $\mathbbm{a}$ and $\mathbbm{b}$ 
are consecutive with respect to $\lex$ if $(\mathbbm{b},\mathbbm{a}) = (\mathbbm{e}_{j+1},\mathbbm{e}_j)$ or $(\mathbbm{b},\mathbbm{a}) = (\mathbbm{e}_{j},\mathbbm{e}_{j+1})$ for some $j$.
We give the following results on admissible $n$-blocks.

\begin{lemma}\label{max_length_of_interval}
    Let $\mathbbm{a}$ and $\mathbbm{b}$ be consecutive blocks of $\D_A^n$ with respect to $\lex$. 
    If $\mathbbm{a}\lex \mathbbm{b}$,
    then $\mathbbm{b}(b)-\mathbbm{a}(b)\leq b^{-n}$.
\end{lemma}

\begin{proof}Let $\mathbbm{a}=a_1a_2\cdots a_n$.
    Suppose $ \mathbbm{b}(b)-\mathbbm{a}(b)> b^{-n}.$
        Let $z=\mathbbm{a}(b)+b^{-n}$ and $\mathbbm{c}=c_1c_2\cdots\in\D^\N$ be the $b$-expansion of $z$.
        Clearly, $\mathbbm{a}\lex\mathbbm{c}$ since $\mathbbm{a}(b)<z$.
        Moreover, $\mathbbm{c}(b)=\mathbbm{a}(b)+b^{-n}<\mathbbm{b}(b)$ and so, $\mathbbm{c}\lex\mathbbm{b}$.
    Since $\mathbbm{a}$ and $\mathbbm{b}$ are consecutive admissible $n$-blocks and $\mathbbm{c}\lex\mathbbm{b}$, then 
    \[c_1c_2\cdots c_n=a_1a_2\cdots a_n.\]
    Thus,
    \begin{align*}
        \mathbbm{a}(b)+b^{-n}&=\mathbbm{c}(b)=\frac{a_1}{b}+\cdots+\frac{a_n}{b^n}+\frac{c_{n+1}}{b^{n+1}}+\frac{c_{n+2}}{b^{n+2}}+\cdots.
    \end{align*}
    Since $\mathbbm{a}(b) = \frac{a_1}{b}+\cdots+\frac{a_n}{b^n}$, we have
    \[1 = \frac{c_{n+1}}{b}+\frac{c_{n+2}}{b^2}+\cdots,\]
    which is a contradiction as the admissibility criterion implies that $c_{n+1}c_{n+2}\cdots(b)<1$.
    Therefore, $\mathbbm{a}\lex\mathbbm{b}\lexeq\mathbbm{c}$ and $\mathbbm{b}(b)-\mathbbm{a}(b)\leq \mathbbm{c}(b)-\mathbbm{a}(b)=z-\mathbbm{a}(b)=b^{-n}$.
\end{proof}

Let $i\in\N\cup\{\infty\}$ be the length of the $b$-expansion of $b-\lfloor b\rfloor$, that is, 
\[i=i_b:=\max\{j\in\N: \text{ the } j\text{th digit of the } b\text{-expansion of } b-\lfloor b\rfloor\text{ is nonzero}\}.\]
Also, we let $K=K_b$ be the maximal length of zero blocks (blocks containing only the digit zero) among the first $i$ digits of the $b$-expansion of $b-\lfloor b\rfloor$.

For $d \in \mathcal{D}$ and $k \in \N$, let $V_k(b;d):=\{ x\in [0,1): d_k(x)= d\}$.
We follow \cite{ZangerTishler2013OnTW} to study the winnability of the cylinder set 
\[C_b[d]:=\{x\in[0,1): d\text{ appears in } \mathbbm{d}(x)\} = \bigcup_{k=1}^\infty V_k(b;d)\]
with respect to the Schmidt $(\alpha, \beta)$-game.

For simplicity, we write $V_k(b;0)=V_k$. 
Now, $V_1 = [0,1/b)$. 
For any $j\in \{0,1,\ldots,\lfloor b\rfloor-1\}$, we see that $j1\in \D_A^2$.
Thus, we have
\[V_2 = 
\begin{cases}
\bigcup_{j=0}^{b-1} \left[\frac{j}{b},\frac{j}{b}+\frac{1}{b^2}\right) & \text{ if } b\in\N \\
\bigcup_{j=0}^{\lfloor b\rfloor-1} \left[\frac{j}{b},\frac{j}{b}+\frac{1}{b^2}\right)
\cup \left[\frac{\lfloor b\rfloor}{b},\min\left\{1,\frac{\lfloor b\rfloor}{b}+\frac{1}{b^2}\right\}\right) & \text{ if }b\notin\N.
\end{cases}\]
For $j \in \mathbb{N}$, let $\phi_j=(j+\sqrt{j^2+4})/2$ be the $j$th metallic mean (that is, the golden, silver, bronze means, etc.).
Note that, if $b \notin \mathbb{N}$,
\[\frac{\lfloor b\rfloor}{b}+\frac{1}{b^2}\leq 1 \iff \lfloor b\rfloor< \phi_{\lfloor b\rfloor}\leq b <\lfloor b\rfloor+1.\]

For the remainder of this section, we let $\c=c_1c_2\cdots=\lim_{\varepsilon\to0^+}\d(1-\varepsilon)$. Let $d'$ be the minimal digit of $\d(1-\varepsilon)$ and $d$ an integer with $0\leq d\leq d'$. 
In general, for $k\in\N$ with $k\geq2$, 
\[V_k(b;d) = \bigcup_{\mathbbm{a}\in\D^{k-1}_A} \Delta(\a d) = \bigcup_{\mathbbm{a}\in\D^{k-1}_A} [\a d(b), u_{\a,d}),\] 
where $\Delta(\a d)$ is the set of real numbers in $[0,1)$ whose first $k$ digits are given by $\a d$, and $u_{\a,d}$ is the the supremum of $\Delta(\a d)$. 
Note that $u_{\a,d} \in (\a d(b), \min \{\a(d+1)(b),1\}]$.

\begin{example}
    Let $b=(1+\sqrt{5})/2$. Then $\D^3_A=\{000,001,010,100,101\}$ and
    \begin{multicols}{2}
    \begin{itemize}
        \item $\Delta(0000)=[0,b^{-4})$
        \item $\Delta(0010)=[b^{-3},b^{-2})$
        \item $\Delta(0100) = [b^{-2},b^{-2}+b^{-4})$
        \item $\Delta(1000)=[b^{-1},b^{-1}+b^{-4})$
        \item $\Delta(1010)=[b^{-1}+b^{-3},1)$
    \end{itemize}
    \end{multicols}
\end{example}

Let \[
E_{k-1,d}:=\{\mathbbm{a}=a_1a_2\cdots a_{k-1}\in \D^{k-1}_A \mid 
a_ja_{j+1}\cdots a_{k-1}(d+1)(b)>1\text{ for some }j<k\}.
\]

\begin{proposition}\label{Eandt}
    Let $2\leq k\in\N$ and $\a\in \D_A^{k-1}$.
    The following are equivalent:
    \begin{enumerate}
        \item[(1)] $u_{\a,d}=\a(d+1)(b)$, i.e., $\Delta(\a d)$ has length $b^{-k}$. 
        \item[(2)] $\a\notin E_{k-1,d}$. 
    \end{enumerate}
\end{proposition}

\begin{proof}
    Let $\a=a_1\cdots a_{k-1}\in \D_A^{k-1}$.
    
    $(1)\Longrightarrow (2)$ Assume $u_{\a,d}=\a(d+1)(b)$ and $\a\in E_{k-1,d}$.
    Then ${a_j}\cdots a_{k-1}(d+1)(b)>1$ for some $j<k$.
    If $j=1$, then $u_{\a,d} = a_1\cdots a_{k-1}(d+1)(b)>1$, a contradiction.
    So $j>1$.
    Let $x=u_{\a,d}-\varepsilon$ for some small $\varepsilon>0$.
    Then $x\in\Delta(\a d)$ and so $\d(x)$ has the form
    \[\d(x)=a_1\cdots a_{k-1} d x_1x_2\cdots.\]
    However, $a_j\cdots a_{k-1} d x_1x_2\cdots(b)\geq 1$, a contradiction.

    $(2)\Longrightarrow (1)$ Suppose $\a\notin E_{k-1,d}$.
    Let $z\in[\a d(b),\a (d+1)(b))$.
    It is enough to show that the first $k$ digits of $z$ are given by $\a d$.
    Observe that $\a (d+1)(b)\leq 1$ and  
    $a_2\cdots a_{k-1}d(b)\leq bz-a_1 < a_2\cdots a_{k-1}(d+1)(b)\leq 1$ since $\a\notin E_{k-1,d}$.
    So $T(z) = bz-a_1$ and the first digit of $z$ is $a_1$.
    Similarly, $a_3\cdots a_{k-1}d(b)\leq bT(z) -a_2< a_3\cdots a_{k-1}(d+1)(b)\leq 1$ since $\a\notin E_{k-1,d}$ and the second digit of $z$ is $a_2$.
    By induction, the first $k$ digits of $z$ is given by $\a d$. 
    This completes the proof.
\end{proof}

By Proposition \ref{Eandt}, we may determine the length of an interval of $V_k$ using the elements of $E_{k-1}$.
Recall that if $u,v,w$ are finite words such that $w=uv$, we say that $u$ is a prefix of $w$ and $v$ is a suffix of $w$.
In other words, $w$ starts with $u$ and ends with $v$.
We also adopt this definition of prefixes and suffixes if $w$ is an infinite word.

\begin{proposition}\label{prefixsuffix}
    Let $k\geq 2$ and  $\a\in E_{k-1,d}$. 
    Then $\a$ has a suffix with positive length that is also a prefix of $\c$.
\end{proposition}

\begin{proof}
    Let $\a=a_1\cdots a_{k-1}\in E_{k-1,d}$.
    Then $a_j\cdots a_{k-1}(d+1)(b)>1$ for some $j<k$.
    So $\a (d+1)$ is not admissible.
    This implies $\c\lex a_\ell\cdots a_{k-1}(d+1)$ for some $\ell<k$.
    However, $ a_\ell\cdots a_{k-1} d\lex\c\lex a_\ell\cdots a_{k-1}(d+1)$, where $a_{\ell}\cdots a_{k-1} d$ is identified with $a_{\ell}\cdots a_{k-1} d 00\cdots$.
    This implies the first $k-\ell$ digits of $\c$ is given by $a_\ell\cdots a_{k-1}$.
\end{proof}

Now, we consider consecutive elements of $\D_A^{k-1}$ which are also elements of $E_{k-1}$.

\begin{lemma}\label{Ekconsecutive}
    Let $k\geq 2$. Let $\a$ and $\b$ be consecutive elements of $\D_A^{k-1}$ such that $\a\lex\b$. 
    If $\a,\b\in E_{k-1,d}$, then $\b$ has the form 
    \[\b = b_1\cdots b_{k'}\underbrace{00\cdots0}_{k-k'-1}\] for some $k'<k-1$ where $b_{k'}\neq0$ and 
    \[\a = b_1\cdots b_{k'-1}(b_{k'}-1)c_1\cdots c_{k-k'-1}.\]
    Moreover, the longest prefix of $\c$ that is also a suffix of $\a$ is $c_1\cdots c_{k-k'-1}$ starting from the $(k'+1)$th digit.
\end{lemma}

\begin{proof}
    Let $\a$ and $\b$ be consecutive elements of $\D_A^{k-1}$ such that $\a\lex\b$ and $\a,\b\in E_{k-1,d}$.
    Let $\b=b_1\cdots b_{k-1}.$

    Suppose $b_{k-1}\neq0$.
    Then $\a=b_1\cdots b_{k-2}(b_{k-1}-1)$.
    Since $\a\in E_{k-1,d}$ and $\b$ is admissible, there exists $j<k$ such that
    \begin{align*}
        b_j\cdots b_{k-2}(b_{k-1}-1)(d+1)(b)&>1>b_j\cdots b_{k-1} d(b) \\ &\geq b_j\cdots b_{k-2}(b_{k-1}-1)(d+1)(b).
    \end{align*}
    Note that $b_j\cdots b_{k-1} d(b)<1$ by the admissibility criterion since the block $b_j\cdots b_{k-1} d$ is also admissible.
    In fact, since $0\leq d\leq d'$ where $d'$ is the minimal digit of $\c$, then this block satisfies the admissibility criterion.
    We have a contradiction. 
    Hence, $b_{k-1}=0$.

    Let $k' = \max\{1\leq j\leq k-2: b_j\neq0\}$.
    In this case, 
    $\a = b_1\cdots b_{k'-1}(b_{k'}-1)c_1\cdots c_{k'-k-1}$.
    Indeed, if $\a'\in \D^{k-1}$ such that $b_1\cdots b_{k'-1}(b_{k'}-1)c_1\cdots c_{k'-k-1}\lex \a'$, then either $\b\lexeq \a'$ or $\a'$ is not admissible.

    Now, suppose $j\leq k'$ such that $ b_j\cdots b_{k'-1}(b_{k'}-1)c_1\cdots c_{k'-k-1}$ is a suffix of $\a$ that is also a prefix of $\c$.
    However, \[b_j \cdots b_{k'-1}b_{k'}0^\infty \lex\c = b_j\cdots b_{k'-1}(b_{k'}-1)c_1c_2\cdots,\]
    a contradiction.
\end{proof}

\begin{proposition}\label{maxconsecutive}
    Let $\a_1\lex \a_2\lex\cdots\lex \a_N$ be consecutive elements of $\D_A^{k-1}$ that are also elements of $E_{k-1,d}$.
    Then $N\leq K+2$.
\end{proposition}

\begin{proof}
    For each $j$, denote $\a_j=a_1^{(j)}\cdots a_{k-1}^{(j)}$.
    By Lemma \ref{Ekconsecutive}, 
    \[\a_2 = a_1^{(2)}\cdots a_{k_2}^{(2)}\underbrace{00\cdots0}_{k-k_2-1}\text{ for some } k_2<k-1\text{ where } a_{k_2}^{(2)}\neq0\]
    and $\a_1=a_1^{(2)}\cdots a_{k_2-1}^{(2)}(a_{k_2}^{(2)}-1)c_1c_2\cdots$.
    Hence, the $(k_2+1)$th digit of $\a_1$ is $c_1\neq0$.
    Also,
    \[\a_3 = a_1^{(3)}\cdots a_{k_3}^{(3)}\underbrace{00\cdots0}_{k-k_3-1}\text{ for some } k_3<k-1\text{ where } a_{k_3}^{(3)}\neq0.\]
    Moreover, by Lemma \ref{Ekconsecutive}, $\a_2 =  a_1^{(3)}\cdots a_{k_3-1}^{(3)}(a_{k_3}-1)c_1\cdots c_{k-k_3-1}$.
    So the $(k_3+1)$th digit of $\a_2$ is $c_1\neq0$ which implies $k_2\geq k_3+1>k_3$.
    Letting $k_j=\max\{1\leq \ell\leq k-1: a_\ell^{(j)}\neq0\}$ for each $j$ yields a sequence 
    \[k-1\geq k_1>k_2>\cdots>k_N.\]
    So $k_N\leq k-N$.

    Moreover, $\a_N$ ends with a zero block of length $k-k_N-1$.
    Since $K$ is the maximum length of zero blocks of $\d(b-\lfloor b\rfloor)$, then $\c$ has at most $K+1$ consecutive zeros.
    By Proposition \ref{prefixsuffix}, $k-k_N-1\leq K+1$.
    Thus, $k-K-2\leq k_N\leq k-N$.
    Therefore, $N\leq K+2$.
\end{proof}

 Hence, $V_k(b;d)$ is composed of intervals of length at most $b^{-k}$ such that if an interval of $V_k(b;d)$ has length less than $b^{-k}$, then by Proposition \ref{maxconsecutive}, there is $j\leq K+2$ such that the $j$th interval of $V_k(b;d)$ to the left of the current interval has length $b^{-k}$.
Also, by Lemma \ref{max_length_of_interval}, the consecutive centers of the intervals of $V_k(b;d)$ are at most $b^{-(k-1)}$ units apart. 

We have the following result.

\begin{theorem}\label{dwinning}
    Suppose $K<\infty$.
    Let $d\in\D$ such that $d\leq d'$ where $d'$ is the minimal digit of $\lim_{\varepsilon\to0^+}\d(1-\varepsilon)$ and $\d(x)$ is the $b$-expansion of $x\in [0,1)$.
    \begin{enumerate}
        \item Let $\alpha,\beta\in(0,1)$ such that $\log_b(\alpha\beta)\notin\Q$. 
    Then $C_b[d]$ is $(\alpha,\beta)$-winning if $\beta>A_b(\alpha)$ where 
   \begin{align}\label{eqn_A}
       A_b(\alpha):=\frac{(2(Kb+2b)+1)\alpha-1}{\alpha[(4(Kb+2b)-1)-\alpha(2(Kb+2b)-1)]}.
   \end{align}
        \item Suppose that \[
   \alpha<\frac{1}{2Kb^2+4b^2+1}.
   \]
   Then $C_b[d]$ is $\alpha$-winning. 
    \end{enumerate}
\end{theorem}

\begin{remark}
    Denoting by $\gamma_1$ and $\gamma_2$ the numerator and denominator of the right-hand side of (\ref{eqn_A}), respectively, we see that $\gamma_2>0$ and $\gamma_2>\gamma_1$. In particular, we have $A_b(\alpha)<1$. Note that Theorem \ref{dwinning} is flexible. In fact, if $\alpha<(2Kb+4b+1)^{-1}$, then $C_b[0]$ is ($\alpha,\beta$)-winning for any $\beta\in (0,1)$ with $\log_b(\alpha\beta)\notin\Q$. 
\end{remark}
\begin{proof}
    (a)
    Suppose $\beta> A_b(\alpha)$. Observe that 
    \begin{align*}
        &\qquad\; 2(Kb+2b)\alpha - \frac{4(Kb+2b)\alpha\beta (1-\alpha)}{1-\alpha\beta}<1-\alpha
        \\&\Longleftrightarrow 2(Kb+2b)\alpha+2(Kb+2b)\alpha^2\beta-4(Kb+2b)\alpha\beta<1-\alpha\beta-\alpha+\alpha^2\beta \\
        &\Longleftrightarrow 2(Kb+2b)\alpha-1+\alpha<\beta(-2(Kb+2b)\alpha^2+4(Kb+2b)\alpha-\alpha+\alpha^2)\\
        &\Longleftrightarrow \beta>A_b(\alpha).
    \end{align*}
   Since $\log_b(\alpha\beta)$ is irrational, there exist $n,k \in\N$ such that 
    \begin{align}\label{eqn:ratio}
    2(Kb+2b)\alpha - \frac{4(Kb+2b)\alpha\beta (1-\alpha)}{1-\alpha\beta}< \frac{K+2}{\rho (\alpha\beta)^nb^{k-1}}<1-\alpha
    \end{align}
    where $\rho>0$ is the initial radius of the $(\alpha,\beta)$-game.
    Hence,
    \begin{enumerate}
        \item[(1)] $\frac{K+2}{b^{k-1}}< \rho(\alpha\beta)^n (1-\alpha)$
        \item[(2)] $2\rho(\alpha\beta)^n\alpha - \frac{4\rho(\alpha\beta)^{n+1} (1-\alpha)}{1-\alpha\beta}< \frac{1}{b^{k}}$.
    \end{enumerate}

    For $j \in \N$, denote the center of Alice's and Bob's $j$-th balls $A_j$ and $B_j$ by $a_j$ and $b_j$, respectively. For the center $a_{n+1}$, 
    Alice must choose so that
    $|a_{n+1}-b_n|\leq \rho (\alpha\beta)^n(1-\alpha)$.
    Let $a$ be the center of an interval of $V_k(b;d)$ that is closest to $b_n$.
    If such interval is of length $b^{-k}$, Alice chooses $a_{n+1}=a$.
    In this case, $|a_{n+1}-b_n|\leq b^{-(k-1)}$.
    If such interval is of length less than $b^{-k}$, Alice chooses $a_{n+1}=a'$, where $a'$ is the center of an interval of $V_k(b;d)$ nearest to $a$ such that the length of the interval of $V_k(b;d)$ centered at $a'$ is $b^{-k}$.
    In this case, $|a_{n+1}-b_n|\leq (K+2)b^{-(k-1)}$ by Proposition \ref{maxconsecutive}.
    Hence, in any case, $|a_{n+1}-b_n|\leq (K+2)b^{-(k-1)}< \rho(\alpha\beta)^n(1-\alpha)$ by (1).

    \begin{figure}
        \centering
        \includegraphics[scale=0.6]{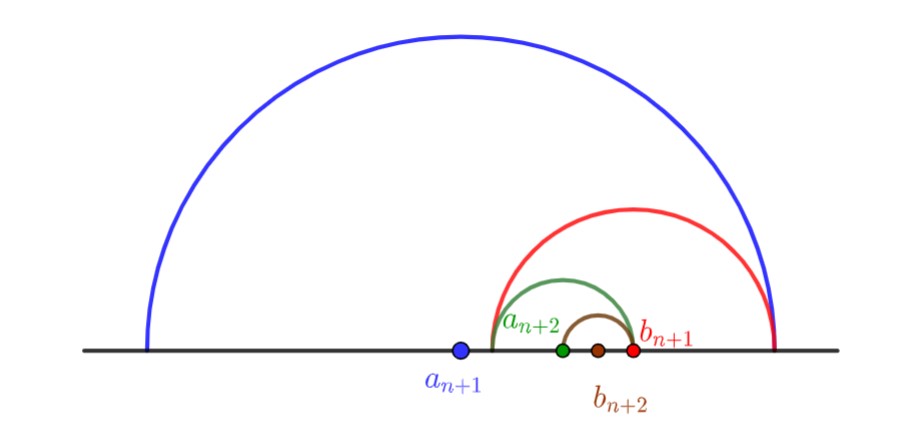}
        \caption{ Centers $a_{n+1},b_{n+1},a_{n+2},b_{n+2}$ of the balls
        $A_{n+1}$, $B_{n+1}$, $A_{n+2}$ and $B_{n+2}$}
        \label{fig:real_schmidt_game}
    \end{figure}
    
    Alice wins if she can force the situation where $A_{n+m}\subseteq V_k(b;d)$ for some $m \in \N$. 
   For avoiding such a situation, Bob's best strategy is to move away as far as possible from $a_{n+1}$ at each turn. See Figure \ref{fig:real_schmidt_game}.
    We may assume that Bob chooses the center $b_{n+m}$ of the ball $B_{n+m}$ to be farthest from the right of $a_{n+m}$, i.e.,
    \[b_{n+m} = a_{n+m}+\rho\alpha (\alpha\beta)^{n+m-1}(1-\beta).\]
    To counter, Alice chooses $a_{n+m+1}$ to be the leftmost possible center, that is, the center of the ball $A_{n+m+1}$ is
    \begin{eqnarray*}
    a_{n+m+1} &=& b_{n+m}-\rho(\alpha\beta)^{n+m}(1-\alpha) \\
    &=& a_{n+m}+\rho(\alpha\beta)^{n+m-1}[\alpha-2\alpha\beta+\alpha^2\beta].
    \end{eqnarray*}
   Alice and Bob play optimally in this manner.
   We have
    \begin{eqnarray*}
        b_{n+1} &=& a_{n+1}+\rho\alpha(\alpha\beta)^n(1-\beta)\\
        a_{n+2} &=& a_{n+1}+\rho(\alpha\beta)^{n}(\alpha-2\alpha\beta+\alpha^2\beta)\\
        b_{n+2}&=&a_{n+1}+\rho(\alpha\beta)^n[\alpha-2\alpha\beta+2\alpha^2\beta-\alpha^2\beta^2]\\
        a_{n+3}&=&a_{n+1}+\rho(\alpha\beta)^n[\alpha-2\alpha\beta+2\alpha^2\beta-2\alpha^2\beta^2+\alpha^3\beta^2].
    \end{eqnarray*}
    In general, 
    \[a_{n+m+1}=a_{n+1}+\rho(\alpha\beta)^n\left[\alpha - 2(1-\alpha)\sum_{j=1}^m(\alpha\beta)^j -\alpha(\alpha\beta)^m\right].\]
    So, the farthest point $c_{n+m+1}$ of $A_{n+m+1}$ from $a_{n+1}$ is its rightmost point, i.e.,
    \[c_{n+m+1} :=a_{n+m+1}+\alpha\rho(\alpha\beta)^{n+m} = a_{n+1}+\rho(\alpha\beta)^n\left[\alpha - 2(1-\alpha)\sum_{j=1}^m(\alpha\beta)^j\right].\]
    Computing for the distance between $c_{n+m+1}$ and $a_{n+1}$, we have 
    \[|c_{n+m+1}-a_{n+1}| = \rho(\alpha\beta)^n\alpha - \frac{2\rho(\alpha\beta)^{n+1}(1-\alpha)[1-(\alpha\beta)^m]}{1-\alpha\beta}.\]
    As $m\to\infty$, the distance $|c_{n+m+1}-a_{n+1}|$ decreases towards \[\rho(\alpha\beta)^n\alpha - \frac{2\rho(\alpha\beta)^{n+1}(1-\alpha)}{1-\alpha\beta}.\]
    By (2), there exists $m\in\N$ such that $|c_{n+m+1}-a_{n+1}|<1/(2b^k)$,
    which implies that $A_{n+m+1}$ is contained in an interval of $V_k$.
    Therefore, $A_{n+m+1}\subseteq V_k(b;d)$.  \\
    
    \noindent
    (b)
It suffices to show that there exists $n,k\in \N$ satisfying (\ref{eqn:ratio}). 
By the remark after Theorem \ref{dwinning}, we may assume that 
\[
 \frac{\log(\alpha^{-1}\beta^{-1})}{\log b}=\frac{t_2}{t_1},
\]
where $t_1,t_2$ are coprime positive integers. Putting $\theta:=b^{1/t_1}$, we have $\alpha\beta=\theta^{-t_2}$ and $b=\theta^{t_1}$. Since $t_1,t_2$ are coprime, we see that
\[
\{(\alpha\beta)^n b^{k-1}\mid n,k\in \N\}
=
\{\theta^m\mid m\in \Z\}.
\]
By assumption, we have $1-\alpha>2b(Kb+2b)\alpha$. From $b\geq \theta$, we get that 
\[
    2(Kb+2b)\alpha - \frac{4(Kb+2b)\alpha\beta (1-\alpha)}{1-\alpha\beta}<\theta^{-1}(1-\alpha). 
\]
Therefore, there exist $n,k\in \N$ satisfying (\ref{eqn:ratio}). 
\end{proof}

\begin{figure}
    \centering
    \includegraphics[scale=0.5]{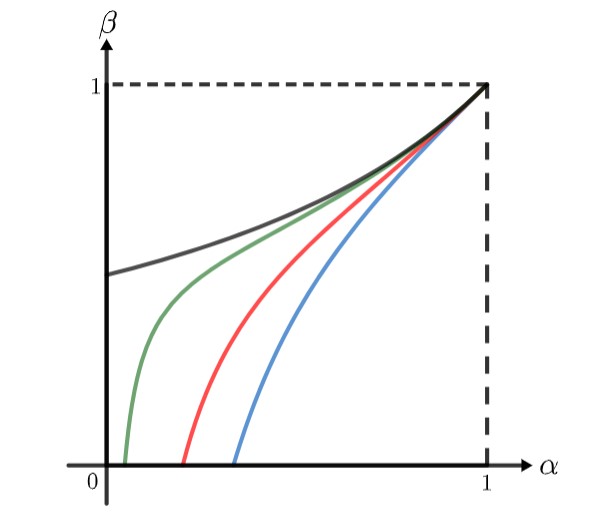}
    \caption{Graph of $\beta=A_b(\alpha)$ for $b=\phi_1$ (blue), $\phi_2$ (red), $\phi_{10}$ (green).
    As $b\to\infty$, the graph of $\beta=A_b(\alpha)$ tends to $\beta = 1/(2-\alpha)$ (black).}
    \label{fig:enter-label}
\end{figure}

\section{Complex expansions}\label{complexsection}

We begin this section by reviewing some literature on the dynamics of the digits of rotational beta expansions in dimension 2. In \cite[Proposition 4.1]{Dem2020},
a criterion that distinguishes admissible blocks from non-admissible ones 
was provided 
for a family of four-fold rotational beta expansions in $\mathbb{R}^2$ by relating the expansions to the so-called \textit{beta Cantor series expansions}. 
In \cite{akiyama2017rotational}, 
necessary and sufficient conditions for the symbolic dynamical system associated
to a rotational beta transformation to be sofic were given.
In \cite{surer2020representations}, an admissibility criterion for a family of rotational beta expansions in dimension 2 called the \textit{zeta expansions} was given. In these zeta expansions, the digits are all rational integers.
In general, determining the admissible blocks of rotational beta expansions in dimension 2 is a nontrivial problem. 

In this section, we consider an expansion on $\C$ with complex base $\xi$ such that $|\xi|>1$.
Let $\L$ be a lattice in $\C$ and $\mathcal{X}\subseteq\C$ be a fundamental domain for $\L$.
Then an analog definition for $(\xi,\L,\mathcal{X})$-expansion (or simply $\xi$-expansion) is obtained from its corresponding rotational beta expansion.

\begin{example}
    Consider $\xi=(1+\sqrt{5})\i/2$ and the lattice $\L=\Z[\i]$ in $\C$ with fundamental domain $\mathcal{X} = \{a+b\i: a,b\in [0,1)\}$.
    Let $z=(5-\sqrt{5})\i/10\in\mathcal{X}$.
    Then $\xi z = -\sqrt{5}/5$ and $-\sqrt{5}/5 +1\in\mathcal{X}$.
    Hence, the first digit of $z$ is $-1$ and 
    $\mathbb{T}_{\xi,\L,\mathcal{X}}(z) = (5-\sqrt{5})/5$.
    Next, $\xi(\mathbb{T}_{\xi,\L,\mathcal{X}}(z))=2\sqrt{5}\i/5\in\mathcal{X}$.
    Thus, $\mathbb{T}^2_{\xi,\L,\mathcal{X}}(z) = 2\sqrt{5}\i/5$ and the second digit of $z$ is $0$.
    Finally, $\xi(\mathbb{T}^2_{\xi,\L,\mathcal{X}}(z))+2 = \mathbb{T}_{\xi,\L,\mathcal{X}}(z)$.
    This means that the third digit of $z$ is $-2$.
    Therefore, the $(\xi,\L,\mathcal{X})$-expansion of $z$ is the periodic expansion $\mathbbm{d}(z)=(-1)\overline{0(-2)}$.
    Indeed, \[z = \frac{-1}{\xi}+\frac{-2}{\xi^3}+\frac{-2}{\xi^5}+\frac{-2}{\xi^7}+\cdots.\]
\end{example}

\begin{example}
Let $\xi = re^{\i\theta}$ where $r>1$ and $\theta\in [0,2\pi)$.
Let $\L=\Z[\i]$ and $\X=\{a+b\i: -1/2\leq a,b< 1/2\}.$
Let $z=x+y\i\in\X$.
Then the first digit $d(z)$ of the $(\xi,\L,\X)$-expansion of $z$ is 
\begin{align*}
  d(z) =&  \left\lfloor\mathrm{Re}(\xi z) +\tfrac{1}{2}\right\rfloor + \left\lfloor\mathrm{Im}(\xi z) +\tfrac{1}{2}\right\rfloor\i \\
  =& \left\lfloor rx\cos(\theta)-ry\sin(\theta) +\tfrac{1}{2}\right\rfloor + \left\lfloor rx\sin(\theta)+ry\cos(\theta) +\tfrac{1}{2}\right\rfloor\i,
\end{align*}
where $\mathrm{Re}(\zeta)$ and $\mathrm{Im}(\zeta)$ are the real and imaginary parts of $\zeta\in \mathbb{C}$, respectively.
\end{example}

Analogous to the one-dimensional case, for $d$ in the digit set $\mathcal{D}$ and $k\in \mathbb{N}$,  we let $V_k(\xi,d):=\{z\in\mathcal{X}: d_k(z)=d\}$ and 
\[C_\xi[d] :=\{z\in\mathcal{X}: d\text{ appears in } \mathbbm{d}(z)\}
=\bigcup_{k=1}^\infty V_k(\xi,d).\]
We aim  to give sufficient conditions  so that the set $C_\xi[0]$ is $(\alpha,\beta)$-winning.

\subsection{Property $(C_k)$}

When $0\in \mathcal{D}$, we consider the scenario where $V_k(\xi;0)$ can be expressed as union of translates of ``scaled-down'' copies of $\mathcal{X}$. This leads to a sufficient condition for which $C_\xi[0]$ is winning (see Theorem \ref{mainthmforcomplex}).
To proceed, for $k\in\N$, we consider the property $(C_k)$ for a $\xi$-expansion.

\begin{definition}
    Let $k\in\N$. 
    We say that the $\xi$-expansion has property $(C_k)$ if 
    \[\xi^{-k}\X+\xi^{-(k-1)}a_{k-1}+\cdots+\xi^{-1}a_1\subseteq\X\]
    for any admissible block $a_1a_2\cdots a_{k-1}\in\D^{k-1}$ with 
    $(C_1):\xi^{-1}\mathcal{X}\subseteq\mathcal{X}$.
\end{definition}

We have the following result.

\begin{proposition}\label{kblockadm} 
    Let $k\in\N $ and $\a=a_1a_2\cdots a_{k}\in\mathcal{D}^{k}$ be an admissible block.
    If $(C_n)$ holds for $n\in\{1,2,\dots, k+1\}$, then $(\a,0)$ is admissible. 
    Moreover, $V_{k+1}(\xi;0)$ can be partitioned into disjoint squares as
    \[ V_{k+1}(\xi;0)=\bigcup_{a_1a_2\cdots a_k \in\D^k\text{ is admissible}}\xi^{-(k+1)}\mathcal{X}+\sum_{j=1}^{k} \xi^{-j}a_j.\]
\end{proposition}

\begin{proof} 
We show the case when $k=1$.
Suppose $(C_1)$ and $(C_2)$ hold. By $(C_1)$, we have
     $\mathcal{X}\supseteq \xi^{-1}\mathcal{X}\supseteq \xi^{-2}\mathcal{X}\supseteq \cdots$.
    Let $a \in\mathcal{D}$.
    The set of points in $\mathcal{X}$ with first digit $a$ and second digit $0$ is $S=(\xi^{-2}\mathcal{X}+\xi^{-1}a)\cap(\xi^{-1}\mathcal{X}+\xi^{-1}a) \cap\mathcal{X}$.
   By $(C_1)$, 
   \[\xi^{-2}\X +\xi^{-1}a = \xi^{-1}(\xi^{-1}\X) +\xi^{-1}a \subseteq \xi^{-1}\X+\xi^{-1}a.\]
    Also, by $(C_2)$, $\xi^{-2}\X +\xi^{-1}a\subseteq \mathcal{X}$.
    Therefore, $S = \xi^{-2}\X +\xi^{-1}a \neq\varnothing$.
    Hence, the block $(a,0)$ is admissible and 
    \[V_2(\xi;0) = \bigcup_{a\in\mathcal{D}}\xi^{-2}\mathcal{X}+\xi^{-1}a.\]
    The cases $k \ge 2$ are similar.
\end{proof}

For the remainder of Section \ref{complexsection}, we only consider $(\xi,\L,\mathcal{X})$-expansions where $\L=\Z[\i]$ and $\mathcal{X} = \{a+b\i \mid a,b\in [-1/2, 1/2)\}$. 
For $\theta\in [0,2\pi]$, define
\begin{eqnarray*}
c(\theta):=&\max\{|\cos(\theta)|,|\sin(\theta)|\}\\
s(\theta):=&\min\{|\cos(\theta)|,|\sin(\theta)|\}.
\end{eqnarray*}

Note that $\theta$ and $\theta+\pi k/2$ $(k\in\Z)$ produce the same results except for a few pathological cases (see remark after Theorem \ref{enumerate_G}).
Moreover, without loss of generality, we restrict $\theta\in[0,\pi/4]$ as we can switch the roles of $\cos(\theta)$ and $\sin(\theta)$ for $\theta\in (\pi/4,\pi/2]$. Thus, $c(\theta)=\cos(\theta)$ and $s(\theta)=\sin(\theta)$. 

\subsection{Square digit sets}

Let $\xi=re^{\i\theta}$ where $\theta\in [0,\pi/4]$ and $r>1$. Let $c=\cos(\theta)$ and $s=\sin(\theta)$.
Let $\L=\Z[\i]$ and $\X = \{a+b\i: a,b\in[-1/2,1/2)\}$.
Since $\theta\in [0,\pi/4]$, we see that 
the digit set $\mathcal{D}$ is the smallest subset $\mathcal{D}'$ of $\L$ for which $\bigcup_{d\in\mathcal{D}'}(\mathcal{X}+d)$ covers $\xi\mathcal{\X}$.\footnote{This is also true when $\X$ is replaced by $\overline{\X}$ if $0\leq \theta\leq\pi/4$.}
Note that this fact does not hold in general when  $\theta\not\in [0,\pi/4]$, for instance, when $\xi=-3$. 
Let $\Dd=\Dd_{(\xi,\L,\X)}:=\bigcup_{d\in\mathcal{D}}(\overline{\mathcal{X}}+d)$.
Then $\Dd$ is a contiguous collection of squares that covers $\xi\overline{\X}$.
We consider the simplest case where $\Dd$ is a rectangle.

\begin{proposition}\label{rect-square}
    If $\Dd$ is a rectangle, then $\Dd$ is a square.
\end{proposition}

\begin{proof}
    Suppose $\Dd$ is a rectangle. 
    Since $0\leq\theta\leq\pi/4$, it follows that $\xi P$  and $P$ lie on the same quadrant in $\mathbb{R}^2$ if $P$ is a corner point of $\overline{\X}$. 
    Moreover, $\xi (1+\i)/2, \xi (1-\i)/2, \xi (-1+\i)/2$ and $\xi (-1-\i)/2$ are the topmost, rightmost, leftmost, and bottommost points of $\xi\overline{\X}$, respectively.
    Hence, for any admissible digit $a+b\i$ with $a,b\in\Z$, we have 
    \[ \left\lfloor \frac{1-r(c+s)}{2}\right\rfloor \le a,b \le  \left\lfloor \frac{1+r(c+s)}{2}\right\rfloor=:N.\]
    Then
    \begin{eqnarray*}
        N\leq &\frac{1+r(c+s)}2&<N+1\\
        -N<&\frac{1-r(c+s)}2&\leq -N+1.
    \end{eqnarray*}
    Note that $(1-rc-rs)/2= -N+1$ if and only if $r(c+s)=2N-1$. If $r(c+s)$ is not an odd integer, then 
    \[\left\lfloor \frac{1-r(c+s)}{2}\right\rfloor = -N.\]
    In this case, $\D = \{a+b\i: a,b\in\Z\text{ such that } |a|,|b|\leq N\}$ and therefore, $\Dd$ is a square. On the other hand, suppose $r(c+s)=2N-1$. 
    Let $z=x+\i y\in \mathcal{X}$ be arbitrarily close to $(1+\i)/2$.
    Then $\mathrm{Im}(d_1(z)) = \left\lfloor r(sx+cy)+1/2\right\rfloor$.
    However,
    \[r(sx+cy)+\frac{1}{2}<\frac{r(c+s)+1}{2} = N.\]
    Hence, 
   \[\max_{d\in\D} \mathrm{Re}(d) = -\min_{d\in\D} \mathrm{Re}(d) = N-1\]
    \[\max_{d\in\D} \mathrm{Im}(d) = -\min_{d\in\D} \mathrm{Im}(d) = N-1.\]
    Thus, there is no admissible digit whose imaginary part is $N$.
    In other words, the corner points of $\xi\overline{\X}$ are on the boundaries (sides) of some adjacent translates of $\overline{\X}$.
    Hence, $\max_{d\in\D} |\mathrm{Re}(d)| = \max_{d\in\D} |\mathrm{Im}(d)| = N-1$. See Figure \ref{odd not odd cases}.
    \end{proof}

    \newpage

   \begin{figure}[ht]
     \centering
     \includegraphics[scale=0.6]{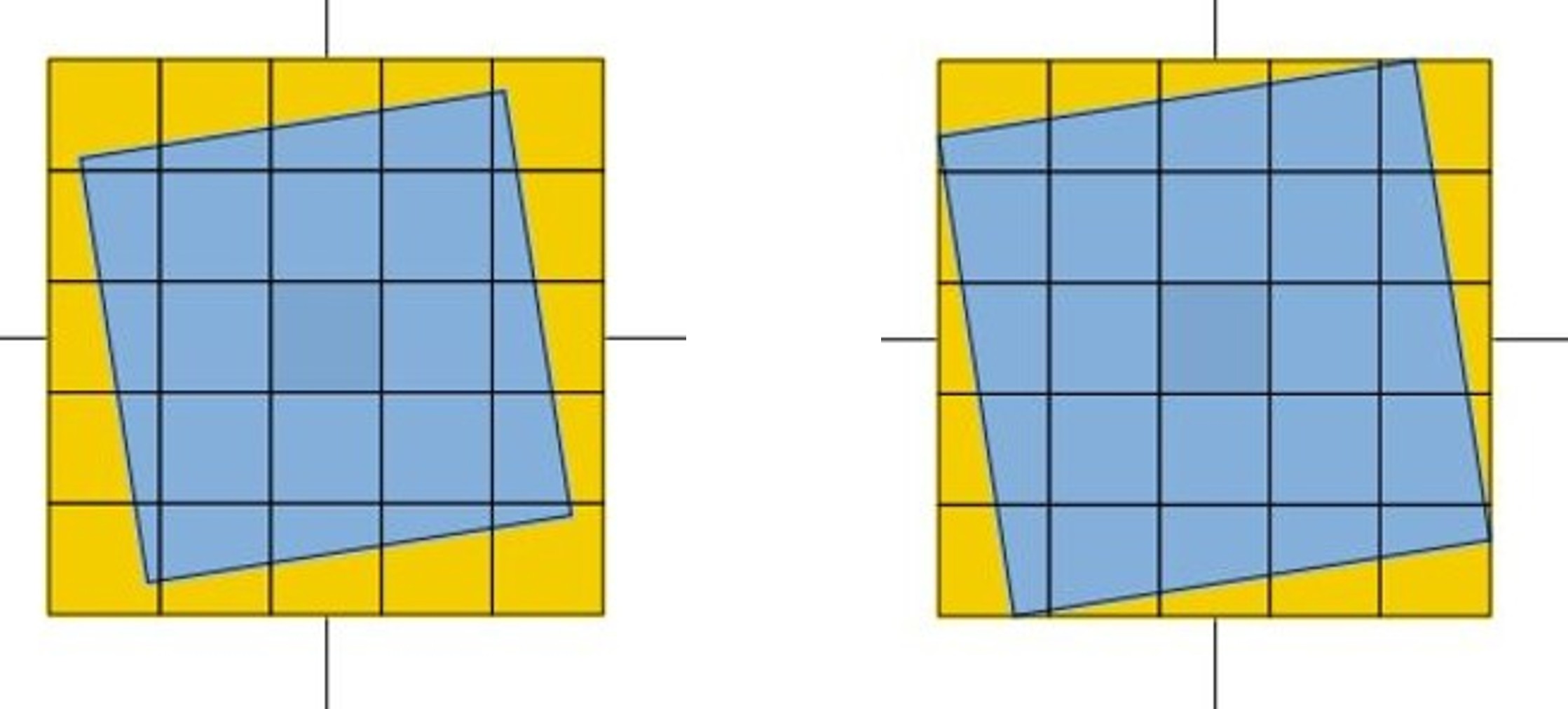}
     \caption{$(re^{\i\theta})\overline{\X}$ inscribed in $\Dd$ where $(r,\theta,N)=\left(\tfrac{5}{c+s}-\tfrac{1}{2},\tfrac{\pi}{20},2\right)$ (left) and $(r,\theta,N)= \left(\tfrac{5}{c+s},\tfrac{\pi}{20}, 3\right)$ (right)}
     \label{odd not odd cases}
\end{figure}

We say that $\xi$ has a square digit set of size $N\in\N$ if 
\[\D = \{a+b\i\mid a,b\in\Z\text{ such that } |a|,|b|\leq N\}.\]
Note that, if $\xi=re^{\i\theta}$ has a square digit set of size $N$, then 
\[N = \left\lceil \frac{r(c+s)+1}{2}\right\rceil-1.\]
The following result provides a criterion to determine if $\xi$ has a square digit set.

\begin{proposition}\label{square}
    The base
    $\xi$ has a square digit set of size $N\in \mathbb{N}$ if and only if 
    \[(2N-1)(c+s)< r\leq \frac{2N+1}{c+s}=:u_N(\theta).\]
\end{proposition}

\begin{proof}
Note that $\xi$ has a square digit set of size $N$ if and only if $|\mathrm{Re}(a)|,|\mathrm{Im}(a)|\leq N$ for all $a\in\D$ and $N+N\i\in\D$. 
In fact, if $N+N\i\in \D$, then for any 
$\varepsilon_1,\varepsilon_2\in \{1,-1\}$, we have $\varepsilon_1 N+\varepsilon_2 N\i\in\D$ by the proof of Proposition \ref{rect-square}.
We have
\begin{align*}
   |\mathrm{Re}(a)|,|\mathrm{Im}(a)|\leq N \text{ for all } a\in\D&\Longleftrightarrow \left\lceil \frac{r(c+s)+1}{2}\right\rceil-1\leq N\Longleftrightarrow r\leq u_N(\theta).
\end{align*}
For the rest of the proof, we assume that $r\leq u_N(\theta)$.
It remains to show that $N+N\i\in\D$ if and only if $r> (2N-1)(c+s)$.

Now, $N+N\i\in\D$ if and only if there exist $x,y\in [-1/2,1/2)$ such that 
\[\left\lfloor rcx-rsy+\frac{1}{2}\right\rfloor =\left\lfloor rsx+rcy+\frac{1}{2}\right\rfloor=N.\]
Equivalently, $(x,y)$ lies on the strips $S_1$ and $S_2$ given by
\[S_1: \frac{c}{s}x-\frac{2N+1}{2rs}<y\leq \frac{c}{s}x-\frac{2N-1}{2rs}\]
\[S_2: \frac{2N-1}{2rc}-\frac{s}{c}x\leq y<\frac{2N+1}{2rc}-\frac{s}{c}x.\]
Define the boundary lines of $S_1$ and $S_2$ as follows:
\[\ell_1: y=\frac{c}{s}x-\frac{2N+1}{2rs}\text{ and } \ell_2: y= \frac{c}{s}x-\frac{2N-1}{2rs}\]
\[\ell_3: y=\frac{2N-1}{2rc}-\frac{s}{c}x\text{ and }\ell_4: y=\frac{2N+1}{2rc}-\frac{s}{c}x.\]
Then $N+N\i\in\D$ if and only if the interior of $S_1\cap S_2\cap \X$ is non-empty.
Indeed, if $x+y\i$ is an interior point of $S_1\cap S_2\cap \X$, then the first digit of $x+y\i $ is $N+N\i$.

Assuming $r\leq u_N(\theta)$, we show that the interior of $S_1\cap S_2\cap \X$ is non-empty if and only if $(2N-1)(c+s)<r$.
We divide the proof into 3 cases. 
When $\theta=\pi/4$, we only consider \textbf{CASES 2} and \textbf{3} as \textbf{CASE 1} is not applicable.\\

\noindent
\textbf{CASE 1:} Suppose $r>(2N-1)/(c-s)$.
Let $\varepsilon>0$ be small so that $z=(1-\varepsilon)(1+\i)/2 \in \X$ is arbitrarily close to $(1+\i)/2$. 
If $d_1(z)= a+b\i$, then 
\[a = \left\lfloor\frac{(1-\varepsilon)(rc-rs)+1}{2}\right\rfloor\text{ and } 
b=\left\lfloor\frac{(1-\varepsilon)(rc+rs)+1}{2}\right\rfloor.\]
Moreover, $a\leq b\leq N$. 
For some $t>0$,
\[r = \frac{2N-1+t}{c-s}.\]  
Choosing $\varepsilon$ sufficiently small, we have
\begin{align*}
    \frac{(1-\varepsilon)(rc-rs)+1}{2}&=\frac{(1-\varepsilon)(2N-1+t)+1}{2}= N+\frac{t-(2N-1+t)\varepsilon}{2}>N.
\end{align*}
Hence, $N\leq a\leq b\leq N$. Thus, $a=b=N$.
This means that $N+N\i\in\D$.\\

\noindent
\textbf{CASE 2:} Suppose $(2N-1)/(c+s)< r\leq (2N-1)/(c-s)$.
Then the corner point $(1+\i)/2$ of $\overline{\X}$ is on the strip $S_2$.
Moreover, the corner point $(-1+\i)/2$ of $\overline{\X}$ is below or on $\ell_3$ 
since $r\leq (2N-1)/(c-s)$. Likewise, $(1-\i)/2$ is below $\ell_3$.

Let $A$ and $B$ be the intersection points of the line $\ell_3$ with the line $y=1/2$ and the line $x=1/2$, respectively. Then
\[A = \frac{2N-1-rc}{2rs}+\frac{1}{2}\i \;\;\;\;\mbox{ and }\;\;\;\; B=\frac{1}{2}+\frac{2N-1-rs}{2rc}\i.\]
Note that $S_1\cap S_2\cap\X$ has an interior point if and only if $A$ is above $\ell_1$ and $B$ is below $\ell_2$. See Figure \ref{A and B}.

 \begin{figure}[ht]
     \centering
     \includegraphics{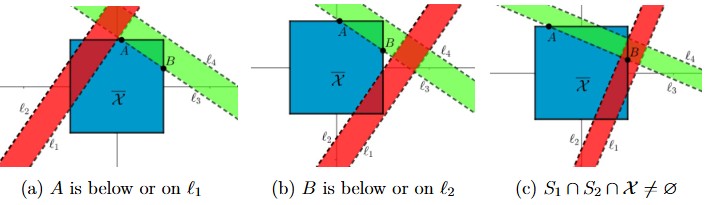}
     \caption{Position of $A$ and of $B$ relative to the strips $S_1$ and $S_2$}
     \label{A and B}
\end{figure}

Observe that
\begin{align*}
   A\text{ is above }\ell_1&\Longleftrightarrow   \frac{c}{s}\cdot\frac{2N-1-rc}{2rs}-\frac{2N+1}{2rs}< \frac{1}{2} 
  \Longleftrightarrow r> 2N(c-s)-(c+s).
\end{align*}
Since $0\leq\theta\le \pi/4$, we have
\[2N(c-s)-(c+s) = \frac{2N(c^2-s^2)-(c+s)^2}{c+s}\leq \frac{2N-(1+2cs)}{c+s}\leq \frac{2N-1}{c+s}<r.\]
So, $A$ is always above $\ell_1$.
It remains to determine when $B$ is below $\ell_2$.
We have
\begin{align*}
    B\text{ is below } \ell_2&\Longleftrightarrow \frac{2N-1-rs}{2rc}< \frac{c}{2s}-\frac{2N-1}{2rs}\\
    &\Longleftrightarrow (2N-1)(c+s)< r.
\end{align*}
Therefore,  $S_1\cap S_2\cap\X$ has an interior point if and only if $r>(2N-1)(c+s)$.\\

\noindent
\textbf{CASE 3:} Suppose $1<r\leq  (2N-1)/(c+s)$.
Observe that
\[r\leq \frac{2N-1}{c+s}\Longleftrightarrow \frac{c+s}{2c}\leq \frac{2N-1}{2rc}\Longleftrightarrow \frac{1}{2}\leq \frac{2N-1}{2rc}-\frac{1}{2} \cdot\frac{s}{c}.\]
This means that the corner point $(1+\i)/2$ of $\overline{\X}$ is below or on $\ell_3$ (see Figure \ref{fig:case3}). 
In this case, $S_2\cap\overline{\X}$, and consequently, $S_1\cap S_2\cap \overline{\X}$, has no interior point.
\end{proof}

\begin{figure}[H]
    \centering
    \includegraphics[scale=0.3]{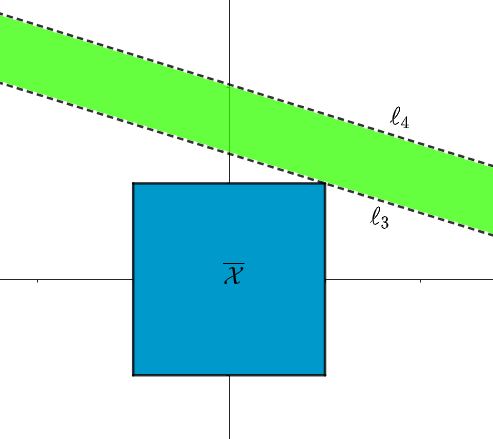}
    \caption{$(1+\i)/2$ is below or on $\ell_3$.}
    \label{fig:case3}
\end{figure}

Given $\theta\in[0,\pi/4]$, there exists $r$ such that $\xi=re^{\i\theta}$ has a square digit set
of size $1, 2, \dots, \lceil (sc+1)/(2sc)\rceil-1$ since $(sc+1)/(2sc)>1$ and
\[(2N-1)(c+s)< \frac{2N+1}{c+s}\iff N< \frac{sc+1}{2sc}.\]

\begin{theorem}
    Let $\theta\in[0,\pi/4]$ and $K(\theta) := \lceil (sc+1)/(2sc)\rceil-1$.
    Then $\xi=re^{\i\theta}$ has a square digit set if and only if 
    \[r\in\bigcup_{N=1}^{K(\theta)} \left.\left((2N-1)(c+s), u_N(\theta)\right.\right].\]
\end{theorem}

\subsection{Condition $(C_k)$ and square digit sets}
Let $\A$ be the set of corner points of $\overline{\X}$.
Then $\xi^{-2}P+a$ is a corner point of $\xi^{-2}\overline{\X}+a$ if and only if $P\in\A$ for any $a\in\D$.
By convexity of $\overline{\X}$,
\begin{align*}
    (C_k) \text{ holds }&\Longleftrightarrow \xi^{-k}\X+\sum_{j=1}^{k-1}\xi^{-j}a_j\subseteq\X,
\end{align*}
where the sum is taken over all admissible blocks $(a_1,a_2,\dots, a_{k-1})\in\D^{k-1}$.
Moreover, by considering the shape of $\X$, this happens if and only if 
\begin{align*}
    \xi^{-k}P+\sum_{j=1}^{k-1}\xi^{-j}a_j\in\overline{\X}\text{ for all }P\in\A.
\end{align*}

For $j\in\{1,2,\dots,k\}$, let $c^{(j)}:=\cos(j\theta)$ and $s^{(j)}:=\sin(j\theta)$. For $P\in\A$,
the real and imaginary parts of $\xi^{-k}P$ are among the numbers 
\[\frac{c^{(k)}+s^{(k)}}{2r^k},-\frac{c^{(k)}+s^{(k)}}{2r^k},\frac{c^{(k)}-s^{(k)}}{2r^k},\text{ and } \frac{s^{(k)}-c^{(k)}}{2r^k}.\]
Thus, 
\[\max\{\textrm{Re}(\xi^{-k}P)\}=\max\{\textrm{Im}(\xi^{-k}P)\}=\frac{|c^{(k)}|+|s^{(k)}|}{2r^k}\]
\[\min\{\textrm{Re}(\xi^{-k}P)\}=\min\{\textrm{Im}(\xi^{-k}P)\}=-\frac{|c^{(k)}|+|s^{(k)}|}{2r^k}.\]

Let $\xi$ have a square digit set of size $N$.
If $a=a_1+a_2\i\in\D$, then 
\[\xi^{-j}a = \frac{c^{(j)}a_1+s^{(j)}a_2}{r^j}+\frac{c^{(j)}a_2-s^{(j)}a_1}{r^j}\i.\]
So,
\[-\frac{N(|c^{(j)}|+|s^{(j)}|)}{r^j}\leq \mathrm{Re}(\xi^{-j}a), \mathrm{Im}(\xi^{-j}a) \leq \frac{N(|c^{(j)}|+|s^{(j)}|)}{r^j}.\]
Therefore, the maximum of $\mathrm{Re}\left(\xi^{-k}P+\sum_{j=1}^{k-1}\xi^{-j}a_j\right)$ 
and $\mathrm{Im}\left(\xi^{-k}P+\sum_{j=1}^{k-1}\xi^{-j}a_j\right)$ 
for $P\in\A$ and admissible block $(a_1,a_2,\dots,a_{k-1})\in \D^{k-1}$ is at most
\[\frac{|c^{(k)}|+|s^{(k)}|}{2r^k}+N\sum_{j=1}^{k-1} \frac{|c^{(j)}|+|s^{(j)}|}{r^j}\] 
while the minimum is at least
\[-\frac{|c^{(k)}|+|s^{(k)}|}{2r^k}-N\sum_{j=1}^{k-1} \frac{|c^{(j)}|+|s^{(j)}|}{r^j}.\]

For example, if $P = (1+\i)/2$, then $\mathrm{Re}(\xi^{-2}P+\xi^{-1}(N+N\i)) = (c^{(2)}+s^{(2)})/(2r^2)+N(c+s)/r$. 

Let $f_{N}^{(k)}(r):=r^k -2N\sum_{j=1}^{k-1}r^{k-j}(|c^{(j)}|+|s^{(j)}|)-(|c^{(k)}|+|s^{(k)}|)$.
Then 
\[f_{N}^{(k)}(r)>0\Longrightarrow (C_k).\]
Note that $f_N^{(k)}$ has a unique positive root, which we call $v_N^{(k)}:=v_N^{(k)}(\theta)$, by the Descartes' rule of signs.
If $k=2$, we have
\[v_N^{(2)}=N(c+s)+\sqrt{N^2(c+s)^2+c^{(2)}+s^{(2)}}.\]
Then, for $\xi$ with a square digit set of size $N$,  $(C_k)$ holds if $r>v_N^{(k)}$.

Observe that $f_N^{(k)}(1)<0$ and so, $v_N^{(k)}>1$.
Also, if $t=v_{N}^{(k-1)}$,
then 
\begin{align*}
    f_N^{(k)}(t) &= t^k-2N\sum_{j=1}^{k-2}t^{k-j}(|c^{(j)}|+|s^{(j)}|)-2Nt(|c^{(k-1)}|+|s^{(k-1)}|) - |c^{(k)}|-|s^{(k)}|\\
    &<t^k-2N\sum_{j=1}^{k-2}t^{k-j}(|c^{(j)}|+|s^{(j)}|)-2Nt(|c^{(k-1)}|+|s^{(k-1)}|)\\
    &<t^k-2N\sum_{j=1}^{k-2}t^{k-j}(|c^{(j)}|+|s^{(j)}|)-t(|c^{(k-1)}|+|s^{(k-1)}|) = t f_N^{(k-1)}(t)=0.
\end{align*}
Thus, $\left\{v_N^{(k)}\right\}_{k\geq2}$ is a strictly increasing sequence.

\begin{theorem}\label{boundforCk}
    Let $\theta\in [0,\pi/4]$ and $r>1$. 
    Let $N\in\N$.
    Then $\xi=re^{\i\theta}$ has a square digit set of size $N$ and $(C_n)$ holds for the $\xi$-expansion for $n\in\{1,2,\dots,k\}$ 
    if $v_N^{(k)}(\theta)<r<u_N(\theta)$.
    In particular, $(C_2)$ holds for the $\xi$-expansion if and only if $v_N^{(2)}(\theta)<r<u_N(\theta)$.
\end{theorem}

\begin{proof}
    Note that $\left\{v_N^{(k)}\right\}_{k\geq2}$ is a strictly increasing sequence and 
    \[v_N^{(k)}>v_{N}^{(2)}> 2N(c+s)>(2N-1)(c+s).\]
    By Proposition \ref{square}, $\xi=re^{\i\theta}$ has a square digit set of size $N$ and so, $(C_n)$ holds for the $\xi$-expansion for $n\in\{1,2,\dots,k\}$ 
    if $v_N^{(k)}(\theta)<r<u_N(\theta)$.
    
    The bound 
    \[\frac{|c^{(k)}|+|s^{(k)}|}{2r^k}-N\sum_{j=1}^{k-1}\frac{|c^{(j)}|+|s^{(j)}|}{r^j}\]
    for $\mathrm{Re}\left(\xi^{-k}P+\sum_{j=1}^{k-1}\xi^{-j}a_j\right)$ and $\mathrm{Im}\left(\xi^{-k}P+\sum_{j=1}^{k-1}\xi^{-j}a_j\right)$ may not be optimal when $c^{(j)}$ or $s^{(j)}$ is negative for some $j\in\{1,\dots,k\}$.
    However, when $k=2$, we have $c,s,c^{(2)},s^{(2)}\geq0$ because $0\leq \theta\leq \pi/4$.
    This implies that 
    \[f_N^{(2)}(r)>0\Longleftrightarrow (C_2)\text{ and so } v_N^{(2)}(\theta)<r<u_N(\theta)\Longleftrightarrow (C_2).\]
    In other words, $(C_2)$ holds for the $\xi$-expansion if and only if $v_N^{(2)}(\theta)<r<u_N(\theta)$.
\end{proof}

For any $N\in \mathbb{N}$, observe that $v_N^{(1)}=c+s<u_N(\theta)$.
Hence, $(C_1)$ holds for any $re^{\i\theta}$-expansion with a square digit set where $r>c+s$. 
Since $c+s\leq \sqrt{2}$, $(C_1)$ holds whenever $r\geq\sqrt{2}$.
Note that the digit set need not be square for $(C_1)$ to hold since $(C_1)$ does not depend on $\D$.

\begin{proposition}\label{cos plus sin}
 Let $\xi=re^{\i\theta}$ where $\theta\in[0,\pi/4]$ and $r>1$. 
    Then $(C_1)$ holds if and only if $r\leq \cos(\theta)+\sin(\theta)$.
    In particular, if $\xi$ has a square digit set, then $(C_1)$ holds.
\end{proposition}

Fix $\theta\in[0,\pi/4]$. Does there exist $v_N^{(k)}<r<u_N$ such that $\xi=re^{\i\theta}$ does not have a square digit set but $(C_k)$ holds?
The answer is negative. 
In other words, if $(C_k)$, then the digit set of the expansion is necessarily square.

\begin{theorem}
   Let $\theta\in[0,\pi/4]$ and $k\in\N$ with $k\geq 2$.
    If $v_N^{(k)}(\theta)<r<u_N(\theta)$ and $(C_k)$ holds, then $\xi=re^{\i\theta}$ has a square digit set.
\end{theorem}

\begin{proof}
    We show the proof for $k=2$.
    Let $N=\left\lceil (rc+rs+1)/2 \right\rceil-1$.
    Then $r\leq u_N(\theta)$ and $|\mathrm{Re}(d)|,|\mathrm{Im}(d)|\leq N$ for any $d\in\D$.
    
   Suppose $r>v_N^{(2)}$ but $\xi=re^{\i\theta}$ does not have a square digit set.
   Then $(C_2)$ holds.
   If $r> (2N-1)/(c-s)$, then $r>(2N-1)(c+s)$. By  Proposition \ref{square}, $\xi$ has a square digit set.
   So, $r\leq (2N-1)/(c-s)$.
   
   Let $\varepsilon>0$ be sufficiently small.
   Let $x=1/2-\varepsilon$ and $y=(rc-2N+1)/(2rs)-c\varepsilon/s$.
   We show that $x+y\i\in\X$. To this end, it is enough to show that $y\in(-1/2, 1/2)$.
  Since $\varepsilon>0$ is sufficiently small, we show that 
   \[-\frac{1}{2}<\frac{rc-2N+1}{2rs}\leq \frac{1}{2}.\]
   Observe that
   \[-\frac{1}{2}<\frac{rc-2N+1}{2rs}\leq\frac{1}{2}
       \Longleftrightarrow -rs<rc-2N+1\leq rs.\]
   The inequality $-rs<rc-2N+1$ follows from $r>(2N+1)/(c+s)>(2N-1)/(c+s)$ (Proposition \ref{square}).
   Meanwhile, the inequality $rc-2N+1\leq rs$ follows from $r\leq (2N-1)/(c-s)$.
   
   Now, the first digit of $x+\i y$ is 
   $N+M\i$ as
   \begin{align*}
   N &= \left\lfloor rcx-rsy+\frac{1}{2}\right\rfloor \\
   M &:=\left\lfloor rsx+rcy+\frac{1}{2}\right\rfloor = \left\lfloor \frac{r-2Nc+c+s}{2s}-r(c+s)\varepsilon\right\rfloor.
   \end{align*}
   Since $\varepsilon>0$ is small, we have 
   \[\frac{r-2Nc+c+s}{2s}\leq M+1.\]
   So, $r\leq 2Nc+2Ms-c+s$. 

   Consider $z=(1-\varepsilon)(1+\i)/2\in\X$.
   By $(C_2)$, $\xi^{-2}z+\xi^{-1}(N+M\i)\in\X$.
   Then $\xi^{-2}z+\xi^{-1}(N+M\i)$ has real part
   \[\frac{c^{(2)}+s^{(2)}}{2r^2}\left(1-\varepsilon\right)+\frac{Nc+Ms}{r}<\frac{1}{2}.\]
   Since $\varepsilon>0$ is small, 
   \[\frac{c^{(2)}+s^{(2)}}{2r^2}+\frac{Nc+Ms}{r}\leq \frac{1}{2}.\]
   Hence, $r^2-2(Nc+Ms)r-(c^{(2)}+s^{(2)}) \ge 0$.
   Then 
   \[(Nc+Ms)+\sqrt{(Nc+Ms)^2+c^{(2)}+s^{(2)}}\leq r\leq 2(Nc+Ms)-(c-s).\]
   Note that $M\geq -N$.
   We have 
   \begin{align*}
       &\qquad\;\sqrt{(Nc+Ms)^2+c^{(2)}+s^{(2)}}\leq  (Nc+Ms)-(c-s)\\
       &\Longrightarrow c^{(2)}+s^{(2)}\leq -2Nc(c-s)-2Ms(c-s)+1-s^{(2)}\\
       &\Longrightarrow 0\leq -2Nc(c-s)-2Ms(c-s)+1-c^{(2)}-2s^{(2)} \\
       &\Longrightarrow 0\leq -2N(c-s)^2+1-c^{(2)}-2s^{(2)} = 
       -2N(1-s^{(2)})+1-c^{(2)}-2s^{(2)}\\
       &\Longrightarrow 
       2s^{(2)}+c^{(2)} \leq 2N(1-s^{(2)})+2s^{(2)}+c^{(2)}\leq 1,
   \end{align*}
   which implies $\theta=0$ and $N=0$ by $0\leq 2\theta\leq {\pi}/2$.
   We have a contradiction because $N\geq 1$ when $\theta=0$ and $\varepsilon>0$ is sufficiently small.
  The proof is similar when $k>2$.
   This completes the proof.
\end{proof}
In what follows, we focus on condition $(C_2)$.
Note that when $\theta\in[0,\pi/4]$,  $(C_2)$ holds if and only if $v_N^{(2)}(\theta)<r<u_N(\theta)$ by Theorem \ref{boundforCk}.
Given $\theta\in [0,\pi/4]$, there exists $r$ such that $(C_2)$ holds for the $\xi= re^{\i\theta}$ (with square digit set of size $N\in\N$) if and only if 
\begin{align*}
    &\qquad\;\; v_N^{(2)}(\theta)<u_N(\theta)\\
    &\Longleftrightarrow N^2(c+s)^4+(c^{(2)}+s^{(2)})(c+s)^2< N^2(1-s^{(2)})^{2}+2N(1-s^{(2)})+1\\
    &\Longleftrightarrow 0> F(N):=4s^{(2)}N^2-2(1-s^{(2)})N+[(c^{(2)}+s^{(2)})(c+s)^2-1].
\end{align*}

Observe that $0>F(x)$ for some $x>0$ if and only if $F$ has positive discriminant, that is, 
\[\Delta:=4(1-s^{(2)})^2-16s^{(2)}[(c^{(2)}+s^{(2)})(c+s)^2-1]>0.\]
This happens if and only if $0\leq \theta<\gamma_1$ where $\gamma_1\approx 0.12988$ is a particular constant.
Suppose $0\leq \theta<\gamma_1$. 
Then there exists $r$ such that $\xi=re^{\i\theta}$ has a square digit set of size $N\in\N$ and $(C_2)$ holds if and only if $L_-<N<L_+$
where
\[L_{\pm} := \frac{2(1-s^{(2)})\pm \sqrt{\Delta }}{8s^{(2)}}\] 
are the roots of $F$.
Note that $L_-<1$.
Also, $L_+>1$ if and only if $0\leq \theta<\gamma_2$ where $0.1249\approx\gamma_2 = 2\tan^{-1}(\delta)$ and $\delta\approx 0.625$ is the smallest positive root of $x^8+16x^7+30x^4-16x+1$.
Therefore, if $0<\theta<\gamma_2$, then there exists $r$ such that $\xi=re^{\i\theta}$ has a square digit set of size $N\in\N$ and $(C_2)$ holds if and only if $N\in \{1,2,\dots,\lceil L_+\rceil-1\}$.

We also have that if $\theta=0$, then $v_N^{(2)}(\theta)=N+\sqrt{N^2+1}<2N+1=u_N(\theta)$. 
Thus, for any $N\in\N$, there exists $r$ such that $\xi=re^{\i\theta}$ has a square digit set and $(C_2)$ holds.

We have the following result. 

\begin{theorem}\label{enumerate_G}
    Let $\theta\in [0,\pi/4)$. Let 
    \[\G(\theta) := \left\{r>1 \mid (C_2)\text{ holds for } \xi=re^{\i\theta}\right\}.\]
    \begin{enumerate}
        \item If $\theta\in [\gamma_2,\pi/4)$, then $\G(\theta)=\varnothing$.
        \item If $\theta\in (0,\gamma_2)$, then $\G(\theta) = \bigcup_{N=1}^{\lceil L_+\rceil-1} (v_N^{(2)}(\theta),u_N(\theta)]$.
        \item If $\theta = 0$, then $\G(\theta) = \bigcup_{N=1}^\infty (N+\sqrt{N^2+1}, 2N+1]$.
    \end{enumerate}
\end{theorem}

\begin{remark}
    If $\theta\in[0,\pi/2)$ and $\pi/2\leq \theta'\leq 2\pi$ such that $\theta' = \theta+ k\pi/2$ for some $k\in\Z$, we obtain parallel results for $\xi=re^{\i\theta}$ and $\xi'=re^{\i\theta'}$, provided their digit sets are the same, since
    since $\xi\overline{\X} = \xi'\overline{\X}$.
    Note that the digit sets are different if and only if $r=u_N(\theta)$ and $\theta'\in[\pi/2,3\pi/2)$.
    In such a case, we take $N=\lfloor (rc+rs+1)/2\rfloor$, instead of $\lceil (rc+rs+1)/2\rceil-1$.
\end{remark}

\subsection{Schmidt game on $C_{\xi}[0]$}
    Let $\xi=re^{\i\theta}$ and $N,k\in\N$ with $k\geq2$ such that $v_N^{(k)}(\theta) < r< u_N(\theta)$.
    Then $\xi$ has a square digit set and $(C_n)$ holds for $n\in\{1,2,\dots,k\}$.
    Recall from  Proposition \ref{kblockadm} that if $(a_1,\dots,a_{k-1})\in\mathcal{D}^{(k-1)}$ is an admissible sequence, then $(a_1,\dots,a_{k-1},0)\in\mathcal{D}^{k}$ is also admissible.

    We introduce a notion of consecutive points in  $\D=\{a+b\i: a,b\in\Z, |a|,|b|\leq N\}$. Write the $(2N+1)^2$ digits as 
    \[w_{s(2N+1)+t}=(-1)^{s+1}(N+1-t)+\i(-N+s),\]
    where $s\in\{0,1,\dots,2N\}, t\in\{1,2,\dots,2N+1\}$.
    Then $|w_{\ell+1}-w_{\ell}|=1$ for any $\ell\in\{1,2,\dots,(2N+1)^2-1\}$. See Figure \ref{order}.

    \begin{figure}[ht]
     \centering
     \includegraphics[scale=0.5]{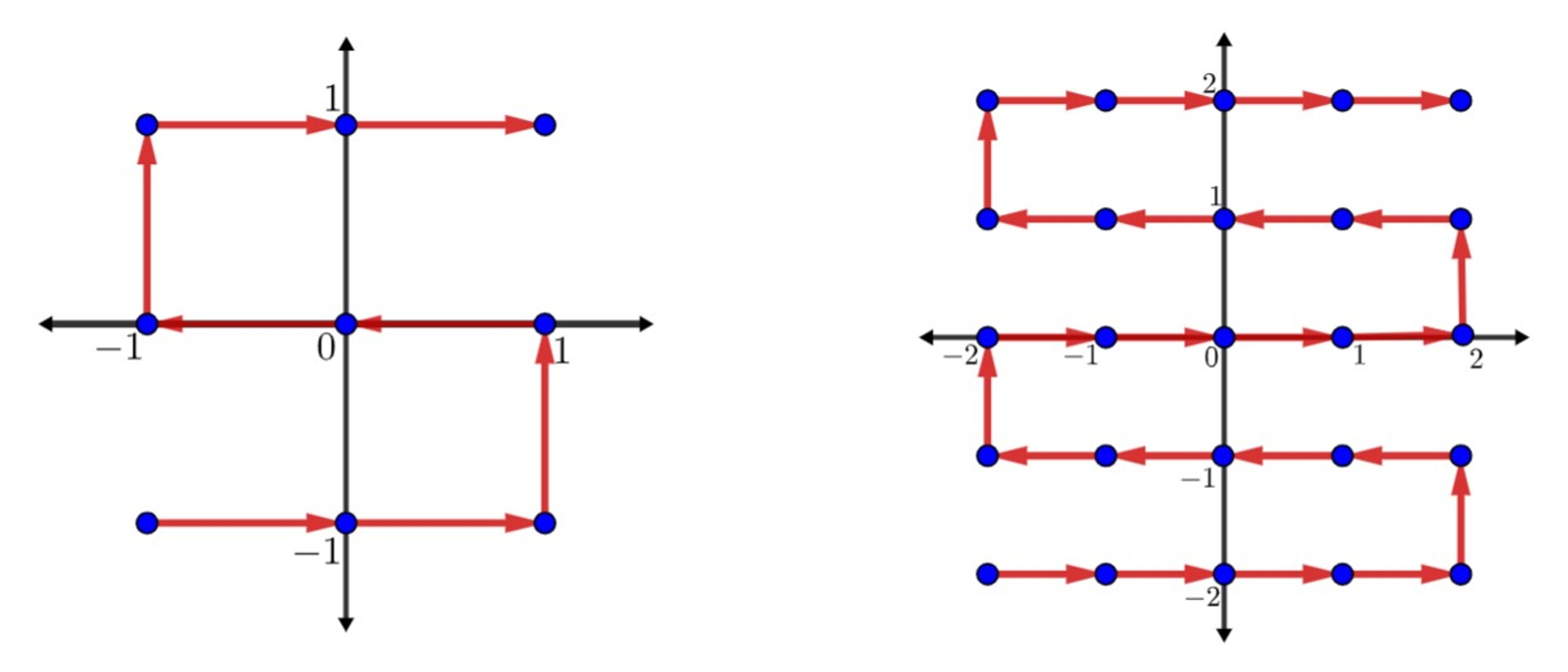}
     \caption{A snake-like ordering on square digit sets of sizes $N=1$ and $N=2$}
     \label{order}
\end{figure}

    This defines a lexicographic ordering for $\D^k$, which, in turns, gives an ordering of points of the form 
    \[a_1\xi^{-1}+a_2\xi^{-2}+\cdots+a_k\xi^{-k}.\]
    This translates to an order of the centers of the squares that make up $V_k$.
    Moreover, the distance between consecutive centers with respect to this order is at most $\sqrt{2}/r^{k-1}$.
    In summary, by Proposition \ref{kblockadm}, $V_k(\xi;0)$ is composed of (rotated) squares with
    \begin{itemize}
        \item area $1/r^{2k}$
        \item radius of biggest circle contained in a square is $1/(2r^k)$
        \item if the first $k-1$ digits coincide, gaps of consecutive centers are at most $\sqrt{2}/r^{k-1}$ wrt to the given order.
    \end{itemize}

    We have the following analog of Theorem \ref{dwinning}.

    \begin{theorem}\label{mainthmforcomplex}
        Let $\xi=re^{\i\theta}$, $\L=\Z[\i]$ and $\mathcal{X}\cong[-1/2,1/2)^2$.
        Let $N,k\in\N$ and suppose $v_N^{(k)}(\theta)< r< u_N(\theta)$. 
        Let $0<\alpha,\beta<1$ and $\rho >0$. 
        Suppose \[\beta>  F_{r}(\alpha):=\frac{(2\sqrt{2}r+1)\alpha-1}{\alpha[(1-2\sqrt{2}r)\alpha +(4\sqrt{2}r-1)]}.\]
        If, either
        \begin{enumerate}
            \item $(2-\alpha)\beta<1$ and  $\rho$ satisfies 
            \begin{align}\label{temporaryname}
\frac{\log_r\rho+\log_r\left(2\alpha-\frac{4 \alpha\beta(1-\alpha)}{1-\alpha\beta}\right)+k}{\log_r(\alpha^{-1}\beta^{-1})}\leq n<  \frac{\log_r\rho+\log_r\left(\frac{1-\alpha}{\sqrt{2}r}\right)+k}{\log_r(\alpha^{-1}\beta^{-1})}
\end{align}
for some $n\in\N$; or,
\item $(2-\alpha)\beta\geq 1$ and $\rho$ is sufficiently large such that 
               \begin{align}\label{temporaryname2}
1<  \frac{\log_r\rho+\log_r\left(\frac{1-\alpha}{\sqrt{2}r}\right)+k}{\log_r(\alpha^{-1}\beta^{-1})},
\end{align}
        \end{enumerate}
        then 
        \[C_\xi[0]:=\{z\in\mathcal{X}: \text{ the }\xi\text{-expansion of }z\text{ has digit 0}\}\]
        is $(\alpha,\beta,\rho)$-winning. 
    \end{theorem}

    \begin{proof}
        Note that $\beta>F_{r}(\alpha)$ if and only if 
    \[2\alpha-\frac{4\alpha\beta(1-\alpha)}{1-\alpha\beta}<\frac{1-\alpha}{\sqrt{2}r}.\]
        Suppose that $\beta>F_{r}(\alpha)$ and $(2-\alpha)\beta<1$. Then (\ref{temporaryname}) holds if and only if 
        \begin{align}\label{inequality}
        2\alpha-\frac{4\alpha\beta(1-\alpha)}{1-\alpha\beta}\leq \frac{1}{\rho(\alpha\beta)^nr^k}<\frac{1-\alpha}{\sqrt{2}r}.
        \end{align}
        Meanwhile, if $(2-\alpha)\beta\geq 1$, then (\ref{inequality}) with $n=1$ follows from the fact that
        $(2-\alpha)\beta\geq 1$ if and only if 
      \[2\alpha-\frac{4 \alpha\beta(1-\alpha)}{1-\alpha\beta}\leq 0\]
      for any positive real number $\alpha,\beta<1$. 
        It follows that
        \begin{enumerate}
            \item[(1)] $\frac{\sqrt{2}}{r^{k-1}}<\rho(\alpha\beta)^n(1-\alpha)$
            \item[(2)] $\rho(\alpha\beta)^n\alpha-\frac{2\rho(\alpha\beta)^{n+1}(1-\alpha)}{1-\alpha\beta}\leq \frac{1}{2r^k}$.
        \end{enumerate}

        By (1), Alice can choose $a_{n+1}$ to be the center of one of the translates of $\X$ that make up $V_{k}(\xi;0)$.
        Bob's best strategy is to move away from $a_{n+1}$ along a common direction, that is, for some unit vector $\vec{v}$.
        \[b_{n+m}=a_{n+m}+\rho\alpha(\alpha\beta)^{n+m-1}(1-\beta)\vec{v}.\] 
        Then Alice responds with 
        \[a_{n+m+1} = b_{n+m}-\rho(\alpha\beta)^{n+m}(1-\alpha)\vec{v}.\]
        Similar to the proof of Theorem \ref{dwinning}, the outcome $\omega$ is at most $1/(2r^k)$ away from $a_{n+1}$. This means that $\omega\in V_k(\xi;0)\subseteq C_\xi[0]$. 
    \end{proof}

    \begin{remark}
    Let $N,k\in\N$ satisfy  $v_N^{(k)}(\theta)< r< u_N(\theta)$. 
    For any $0<\alpha,\beta<1$ with $\beta>F_{r}(\alpha)$ and $(2-\alpha)\beta<1$, we see that there exists $\rho>0$ such that (\ref{temporaryname}) holds for some $n\in\N$. 
    Indeed, (\ref{temporaryname}) holds if and only if 
    \begin{align*}
        \left\lceil \frac{\log_r\rho}{\log_r (\alpha^{-1}\beta^{-1})}+\Phi \right\rceil <
        \left\lceil \frac{\log_r\rho}{\log_r (\alpha^{-1}\beta^{-1})}+\Psi \right\rceil \mbox{ and } 
        2\leq \left\lceil \frac{\log_r\rho}{\log_r (\alpha^{-1}\beta^{-1})}+\Psi \right\rceil, 
    \end{align*}
    where 
    $\Phi$ and $\Psi$ are positive real numbers independent of $\rho$ with $\Phi<\Psi$.   
   \end{remark}

    \begin{example}
        Let $\xi=re^{\i\theta}$ where $r=9/2$ and $0\leq \theta\leq 0.064$. 
        Then $v_2^{(2)}(\theta)<r<5=u_2(\theta)$.
        So, ($C_2$) holds for the $(\xi,\L,\X)$-expansion where $\L=\Z[\i]$ and $\X=\{a+b\i: -1/2\leq a,b<1/2\}$.
        Let $\rho=2$, $\alpha=3/5$ and
        \[1 > \beta > F_{9/2}(3/5) = \frac{8495-180\sqrt{2}}{11901}\approx 0.6924.\]
        Then we satisfy the conditions of Theorem \ref{mainthmforcomplex} with $n=1$ and $k=2$.
        Hence, $C_{\xi}[0]$ is $(\alpha,\beta,\rho)$-winning.
    \end{example}

\section{Quaternion expansions}
The skew field (i.e. noncommutative division ring) $\mathbb{H}$ of real quaternions is a unital associative algebra over $\mathbb{R}$ with basis $\{1, \i, \j, \k\}$ where $\i, \j, \k$ are \textit{imaginary} units satisfying 
\[\i^2=\j^2=\k^2 = \i\j\k =-1.\]
A real quaternion $q\in \mathbb{H}$ is written uniquely as $q=a+b\i+c\j+d\k$, where $a, b, c, d \in \mathbb{R}$. 
We denote by $\mathbb{H}_H$ and $\mathbb{H}_L$ the rings of
Hurwitz quaternions and Lipschitz quaternions, respectively:
\begin{align*}
   \mathbb{H}_H &:= \{a+b\i+c\j+d\k \mid a, b, c, d \in \mathbb{Z} \mbox{ or }  a, b, c, d \in \mathbb{Z}+1/2\}\\
   \mathbb{H}_L &:= \{a+b\i+c\j+d\k \mid a,b,c,d \in \mathbb{Z}\}.
\end{align*}

In this section, we consider a 4-dimensional rotational beta expansion where the matrix parameter
$M$ is isoclinic, that is, $M\in SO(4)$ and there exists a unit quaternion $p$ such that $Mx=px$ for all $x\in\R^4$ when we view $\R^4$ as  the set $\H$ of real quaternions.
It has been shown in \cite{ojmpreprint} that the expansion in this setting corresponds to an expansion on the set $\H$ of real quaternions.

    Let $q\in\H$ with $|q|>1$. 
    Let $\L$ be a point lattice on $\H$ with a fundamental domain $\mathcal{X}$ such that $0\in\L\cap \mathcal{X}$. Then 
    \[\mathbb{H} =  \bigcup_{d\in \L} (\mathcal{X}+d).\]
    For $z\in\mathcal{X}$, there exists a unique $d(z)\in\L$ such that $qz-d(z)\in\mathcal{X}$.
    We define the $(q,\L,\mathcal{X})$-transformation $\mathbb{T}=\mathbb{T}_{q,\L,\mathcal{X}}: \mathcal{X}\to\mathcal{X}$ by 
    \[\mathbb{T}(z) = qz-d(z).\]
    The $(q,\L,\mathcal{X})$-expansion of $z\in\mathcal{X}$ is
     $z=q^{-1}d_1+q^{-2}d_2+\cdots$,
     where the $j$th digit $d_j$ of the expansion is given by $d(\mathbb{T}^{j-1}(z)) \in \L$. 
     As before, we write $\mathbbm{d}(z) = \mathbbm{d}_q(z):=d_1d_2\cdots$ and
     $\mathcal{D}=\mathcal{D}(q, \mathcal{L}, \mathcal{X}):=\{d(z) \mid z \in \mathcal{X}\}$ for the 
     digit set of the expansion. 
     
     We give some families of $q$-expansions.

    \begin{example}[Rotational beta expansion]\label{q-expansion}
    Let $\{\eta_1,\eta_2,\eta_3,\eta_4\}$ be an $\mathbb{R}$-basis of $\mathbb{H}$. Then
    $\mathcal{L}:=\mathbb{Z}\eta_1+\mathbb{Z}\eta_2+\mathbb{Z}\eta_3+\mathbb{Z}\eta_4$ 
    is a lattice of $\mathbb{H}$ with fundamental domain 
    $\mathcal{X}:=\{t_1\eta_1+t_2\eta_2+t_3\eta_3+t_4\eta_4 \mid  t_1,t_2,t_3,t_4\in [0,1)\}.$
    Let $q\in\H$ with $|q|=\beta>1$ and  $q/\beta:=\beta^{-1}q = a+b\i+c\j+d\k$ where $a,b,c,d\in\R$. 
    If \begin{equation*}\label{eqn:eqn1}
        M = \begin{bmatrix}
        a&-b&-c&-d\\ b&a&-d&c\\ c&d&a&-b\\ d&-c&b&a 
    \end{bmatrix}\in SO(4) \tag{$\star$}
    \end{equation*}
    (i.e., $M$ is isoclinic), then $\beta Mx = qx$ for any $x\in\H$ when $x$ and $qx$ are viewed as elements of $\R^4$.
    Hence, the $(q,\L,\mathcal{X})$-expansion (see \cite{ojmpreprint}) corresponds to the 4-dimensional rotational beta expansion with parameters $\beta$ and $M$.

    Conversely, if $M$ has the form $(\star)$
    where $a,b,c,d\in\R$ such that $a^2+b^2+c^2+d^2=1$ and $\beta>1$, then the rotational beta expansion with parameter $(\beta,M)$ corresponds to the $(q,\L,\mathcal{X})$-expansion where $q=\beta(a+b\i+c\j+d\k)\in\H$.

    For instance, let $(\eta_1,\eta_2,\eta_3,\eta_4)=(1,\i,\j,\k)$. 
    Let $q=(1+\sqrt{5})\i/2$.
    Then the $q$-expansion of $(1+\j)/2$ is given by the purely periodic expansion 
    \[ \overline{0\; (-2-2\j)\; (\i+\k)\; (-1-\j)\; (\i+\k)\; (-1-\j)}.\]
\end{example}

\begin{example}[Zeta expansion] \label{zetaexpansion}
    Let $\zeta\in\H\setminus\R$ with  $|\zeta|>1$.
    Let $\eta\in\H$ such that $\mathrm{Re}(\eta)=0$ with $|\eta|=1$ and $\eta\cdot\zeta=0$ where $\cdot$ is the usual dot product on $\R^4$.
    Let $\varepsilon\in[0,1)$.
    Consider the lattice 
    \[\L = \mathbb{Z}+\mathbb{Z}\overline{\zeta}+\mathbb{Z}\eta+\mathbb{Z}\overline{\zeta}\eta\]
    with fundamental domain
    \[\mathcal{X} = \{a_1+a_2\overline{\zeta}+a_3\eta+a_4\overline{\zeta}\eta \mid -\varepsilon\leq a_1,a_2,a_3,a_4<1-\varepsilon\}.\]
    Then the digit set $\mathcal{D}$ of the $\zeta$-expansion is contained in $\Z+\Z\eta$ (see \cite{ojmpreprint}). 
    This expansion extends  $\zeta$-expansion on complex numbers \cite{surer2020representations} to $\mathbb{H}$.
\end{example}

\begin{example}[Symmetric $q$-expansion]\label{symmetric}
    Let $q\in\H$ with $|q|>1$.
    For $1\le i \le 4$, let $\varepsilon_i>0$. 
    Let 
    $\mathcal{X}=\{a_1+a_2\i+a_3\j+a_4\k \mid a_i\in[-\varepsilon_i,\varepsilon_i)\}$
    be a fundamental domain of a point lattice $\L$ be in $\H$ such that $0\in\L$. 
    We say that the $q$-expansion is $(\varepsilon_1,\varepsilon_2,\varepsilon_3,\varepsilon_4)$-symmetric.
    If $\varepsilon_1=\cdots=\varepsilon_4=\varepsilon$, then the $q$-expansion is $\varepsilon$-symmetric.
    For instance, if $\L=\mathbb{H}_L$ with fundamental domain $\mathcal{X}\cong[-1/2,1/2)^4$, then the $q$-expansion is $1/2$-symmetric.
    If $\L=\mathbb{H}_H$ and $\mathcal{X} \cong [-1/2,1/2)^3\times [-1/4,1/4)$, then the $q$-expansion is $(1/2,1/2,1/2,1/4)$-symmetric.
    
\end{example}

    \subsection{$q$-expansions and Schmidt game}
    Let us consider the $(q, \mathcal{L}, \mathcal{X})$-expansion on $\mathbb{H}$ where
     $q\in\H$ with $|q|>1$, 
     $\L$ is a point lattice in $\H$ with a bounded fundamental domain $\mathcal{X}$. 
     For an admissible block $\Omega=(a_1, a_2, \dots, a_n)$,  define the set 
     \[\mathcal{C}[\Omega]=\mathcal{C}_{(q,\L,\mathcal{X})}[\Omega] := \{z\in\mathcal{X}\mid  \Omega \text{ appears in the expansion } \mathbbm{d}_q(z)\}.\] 
    
     \subsubsection{$(\alpha, \beta)$-losing}
     In this section, we give values of $(\alpha,\beta)\in (0,1)^2$ for which $\mathcal{C}[\Omega]$ is not $(\alpha,\beta)$-winning with respect to the Schmidt game played in $\mathcal{X}$. 
     This means that Bob can employ a strategy so that $\Omega$ is not a block of the outcome $\omega$. 
     In \cite{ZangerTishler2013OnTW}, such a set is called $(\alpha,\beta)$-losing.
     Note that, for admissible blocks $\Omega$ and $\Omega'$ such that $\Omega$ is a sub-block of $\Omega'$, if ${\mathcal C}{[\Omega]}$ is $(\alpha,\beta)$-losing, then ${\mathcal C}{[\Omega']}$ is $(\alpha,\beta)$-losing. 
    In general, not being $(\alpha,\beta)$-winning is a necessary but not sufficient condition for being $(\alpha,\beta)$-losing when the target set is not Borel \cite{FISHMAN_LY_SIMMONS_2014}. 
     In this section, we only consider finite unions of Borel sets which are in turn Borel.
     Hence, we do not distinguish between losing and not winning sets.

 \begin{theorem}\label{notwinning}
    Assume that $\X$ has a nonzero interior point.
   There exists a positive constant $C_{\mathcal{X}}$ satisfying the following: 
    Let $q\in\H$ with $|q|>1$ and let $\Omega=d_1d_2\cdots d_n \in \mathcal{D}^n$ be a $q$-admissible word of length $n$ such that 
    \[
    C_{\Omega}:=
    C_{\mathcal{X}}+C_{\mathcal{X}}
    \left|
    \sum_{j=1}^n q^{n-j} d_j
    \right|<|q|^n.
    \]
    For any real number $\alpha<1$ so that $C_{\Omega}|q|^{-n}\leq \alpha$ and $|q|^{-n}<\alpha$, we have $\mathcal{C}_q[\Omega]$ is $(\alpha,\beta)$-losing, 
    where $\beta=\alpha^{-1}|q|^{-n}<1$.
 \end{theorem}

\begin{remark}
    Note that $\mathcal{C}_q[\Omega]$ is ($\alpha,\beta$)-losing if there exists $\rho>0$ such that $\mathcal{C}_q[\Omega]$ is not ($\alpha,\beta$)-winning. More precisely, we can prove that $\mathcal{C}_q[\Omega]$ is not ($\alpha,\beta$)-winning for any $\rho<\rho_0(\mathcal{X})$, where $\rho_0(\mathcal{X})$ is a positive constant depending only on $\mathcal{X}$. 
\end{remark}

\begin{proof}
        We show the case where $n=1$. Let $\Omega=d\in\mathcal{D}$.
    Let  $\xi\neq0$ be an interior point of $\mathcal{X}$.
    Let $\rho>0$ such that $|\xi|>2\rho$ and $\overline{B(\xi,\rho)}\subseteq\mathcal{X}$.
    Let $D:=\sup_{z\in\mathcal{X}}|\xi-z|$ and $M:=\sup_{z\in\mathcal{X}}|z|>0$. 

    Let 
    \[
    C_{\mathcal{X}}:=\max\left\{
    1+\frac{D}{\rho}, \frac{M}{|\xi|-2\rho}, 
    \frac{1}{|\xi|-2\rho}
    \right\}.
    \]
    Note that 
    \[C_d\geq \max\left\{1+\frac{D}{\rho}, \frac{M+|d|}{|\xi|-2\rho}\right\}.\]
    
    Assume that $C_d/|q|\leq \alpha$ and $\alpha\beta= 1/|q|$.
    Denote by $a_k$ and $b_k$ the centers of Alice's and Bob's $k$th ball, respectively.
    Let Bob choose the initial ball $B_0=\overline{B(\xi,\rho)}$, i.e., $b_0=\xi$.
    Then Alice chooses the ball $A_1=\overline{B(a_1,\rho\alpha)} \subseteq B_0$.
    For $k,m\in\N$, let $a_k^{(m)}$ be the $m$th digit in the $q$-expansion of $a_k$.

    We let Bob use the strategy where he chooses $B_k=\overline{B(b_k,\alpha^k\beta^k\rho)}$ for $k\in\N$ where 
    \begin{align}\label{eqn:def_bk}
    b_{k}=q^{-1}a_1^{(1)}+q^{-2}a_2^{(2)}+\cdots+q^{-k}a_{k}^{(k)}+q^{-k}\xi.
    \end{align}
    We show that this strategy works by inductively showing the following:
    \begin{itemize}
        \item $a_{k}^{(k)}\neq d$
        \item $b_{k}$ is a valid choice, that is, $B_{k}\subseteq A_k$
        \item for any $z\in B_k$, the first $k$ digits of the $q$-expansion of $z$ are $a_1^{(1)},a_2^{(2)},\dots,a_k^{(k)}$, in this order.
    \end{itemize}
    
    For the base case, we claim that $a_1^{(1)}\neq d$.
    Suppose otherwise. 
    Then
    $qa_1-d\in \mathcal{X}$.
    So, $|qa_1-d|\leq M$ and
    \[|a_1|\leq \frac{M+|d|}{|q|} = (|\xi|-2\rho)\frac{M+|d|}{(|\xi|-2\rho)|q|}\leq(|\xi|-2\rho) \frac{C_d}{|q|}\leq (|\xi|-2\rho)\alpha.\]
    
    Let $z_1\in A_1$.
    Then $|z_1|\leq|a_1|+\rho\alpha \leq (|\xi|-\rho)\alpha< |\xi|-\rho$. 
    However, since $z_1\in A_1\subseteq B_0=\overline{B(\xi,\rho)}$, we have
    $|\xi|-|z_1|\leq |z_1-\xi|\leq \rho$.
    Contradiction. 
    
    Now, Bob chooses the ball $B_1=\overline{B(b_1,\alpha\beta\rho)}$, where $b_1$ is defined by (\ref{eqn:def_bk}).
    If $z\in B_1$, then
     $|z-b_1|=|z-q^{-1}(\xi+a_1^{(1)})|\leq \alpha\beta\rho\leq \rho/|q|$.
    Since $qa_1-a_1^{(1)}\in\mathcal{X}$, we have $|qa_1-a_1^{(1)}-\xi|\leq D$.
    Observe that
    \begin{align*}
        |z-a_1|&\leq |z-q^{-1}(\xi+a_1^{(1)})|+|a_1-q^{-1}(\xi+a_1^{(1)})|
        \\&\leq \frac{\rho}{|q|}+\frac{|qa_1-a_1^{(1)}-\xi|}{|q|}
        \\&\leq \frac{\rho}{|q|}+\frac{D}{|q|} = \frac{\rho}{|q|}\left(1+\frac{D}{\rho}\right).
    \end{align*}
    So, $|z-a_1| \leq \rho C_d/|q| \leq \rho\alpha$.
    Then $B_1\subseteq A_1$.

    Next, we show that the first digit of any $z\in B_1$ is $a_1^{(1)}\neq d$.
In fact, since $|qz-a_1^{(1)}-\xi|\leq \rho\alpha\beta|q|=\rho$ because $z\in B_1$, we get $qz-a_1^{(1)}\in \overline{B(\xi,\rho)}\subseteq \mathcal{X}$.

So, we assume that Bob's strategy works for the first $k$ turns for some $k\in\N$. 
In particular, after Alice chooses $a_{k+1}$, 
Bob responds with the ball $B_{k+1}=\overline{B(b_{k+1},\alpha^{k+1}\beta^{k+1}\rho)}$.
Since $a_{k+1}\in A_{k+1}\subseteq B_k$, we see 
by the third inductive hypothesis that 
\begin{align}\label{eqn:rel_akandaj}
a_{k+1}^{(j)}=a_{j}^{(j)}\ne d \ \mbox{for any $j\in\{1,\dots,k\}$.}
\end{align}
    
    We show that $a_{k+1}^{(k+1)}\neq d$.
    Suppose otherwise.
    Since $a_{k+1}\in A_{k+1}\subseteq B_{k} = \bigcap_{j=1}^k B_j$, we can write
    \[a_{k+1} = \sum_{j=1}^{k}q^{-j}a_j^{(j)}+q^{-k}t_{k} \text{ for some } t_{k}\in \X.\]
    Since $a_{k+1}^{(k+1)}=d$, it follows that $qt_k-d\in\X$. 
    Thus, $|t_k|\leq (M+d)/|q|\leq (|\xi|-2\rho)\alpha$ as in the base case.
    Consider any point $z_{k+1}$ of the form $z_{k+1}=\sum_{j=1}^k q^{-j}a_j^{(j)}+q^{-k}s_k$ where $|t_k-s_k|\leq \alpha\rho$.
    Then 
    \[|z_{k+1}-a_{k+1}| = |q|^{-k}|t_k-s_k|\leq |q|^{-k}\alpha\rho = \alpha^{k+1}\beta^k\rho.\]
    Hence, $z_{k+1}\in A_{k+1}\subseteq B_{k}$ and 
    \[|s_k|\leq \alpha\rho+|t_k|\leq \alpha\rho+(|\xi|-2\rho)\alpha= \alpha(|\xi|-\rho)<|\xi|-\rho.\]
    But $z_{k+1}\in B_k = \overline{B(b_k,\alpha^k\beta^k\rho)}$ implies 
    \[\alpha^k\beta^k\rho\geq |b_k-z_{k+1}| = |q|^{-k}|\xi-s_k| = \alpha^{k}\beta^{k}|\xi-s_k|.\]
    Thus, $|s_k|\geq |\xi|-\rho$.
    We have a contradiction.
    Hence, $a_{k+1}^{(k+1)}\neq d$.

    Now, we show that $B_{k+1}\subseteq A_{k+1}$ and the $(k+1)$th digit in the $q$-expansion of any element of $B_{k+1}$ is $a_{k+1}^{(k+1)}\neq d$.
    Let $z\in B_{k+1}$.
    Then 
    \[|z-b_{k+1}|\leq\alpha^{k+1}\beta^{k+1}\rho = \rho/|q|^{k+1}.\]
    Since $a_{k+1}\in B_k$, by (\ref{eqn:rel_akandaj}) we have 
    \[a_{k+1}=\sum_{j=1}^{k+1}q^{-j}a_j^{(j)}+q^{-(k+1)}t_{k+1}\]
    for some $t_{k+1}\in\X$.
    Thus, $|b_{k+1}-a_{k+1}| = |q|^{-(k+1)}|\xi-t_{k+1}| \leq |q|^{-(k+1)}D$. Moreover,
    \begin{align*}
        |z-a_{k+1}|&\leq |z-b_{k+1}|+|b_{k+1}-a_{k+1}|\\
        &\leq |q|^{-(k+1)}\rho+|q|^{-(k+1)}D\\
        &=\frac{\rho}{|q|^{k+1}}\left(1+\frac{D}{\rho}\right)\\
        &\leq \frac{C_d\rho}{|q|^{k+1}}=\frac{C_d\alpha^{k}\beta^{k}\rho}{|q|}\\
        &\leq \alpha^{k+1}\beta^{k}\rho.
    \end{align*}
    Hence, $z\in A_{k+1}$ and so, $B_{k+1}\subseteq A_{k+1}\subseteq B_k$.

    Finally, we show that the $(k+1)$th digit of any $z\in B_{k+1}$ is $a_{k+1}^{(k+1)}\neq d$.
    Since $z\in B_{k+1}\subseteq B_k$, then 
    \[z=\sum_{j=1}^{k}q^{-j}a_j^{(j)}+q^{-k}u_k\]
    for some $u_k\in\X$.
    Now, since $z\in B_{k+1}$, we have
    \begin{eqnarray*}
        \alpha^{k+1}\beta^{k+1}\rho &\geq& |z-b_{k+1}|\\
        &=& |q|^{-k}|q^{-1}(a_{k+1}^{(k+1)}+\xi)-u_k|\\
        &=&\alpha^k\beta^k|q^{-1}(a_{k+1}^{(k+1)}+\xi)-u_k|.
    \end{eqnarray*}
    Therefore, 
    \[|(qu_k-a_{k+1}^{(k+1)})-\xi|\leq \alpha\beta\rho|q| = \rho.\]
    So, 
    \[qu_k-a_{k+1}^{(k+1)}\in \overline{B(\xi,\rho)}\subseteq\X.\]
    Thus, the $(k+1)$th digit of $z$ is $a_{k+1}^{(k+1)}\neq d$.

    Therefore, the outcome $\omega$ has $q$-expansion $\mathbbm{d}(\omega) = a_1^{(1)}a_2^{(2)}\cdots$ where $d$ does not appear. 
    In other words, $\mathcal{C}_q[d]$ is $(\alpha,\beta)$-losing.
    
    In the case where the length of $\Omega$ is a general positive integer, we can show Theorem \ref{notwinning} in the same way as above.
\end{proof}

Theorem \ref{notwinning} can be extended to the setting of rotational beta expansions (including real and complex expansions) in arbitrary dimension $m\ge 1$ as follows. Let $q > 1$ be a real number and let $\Theta \in SO(m)$. The statement of Theorem \ref{notwinning} is modified by replacing each term $q^{n-j} d_j$ with $q^{n-j} \Theta^{n-j} d_j$.
Furthermore, the proof of Theorem \ref{notwinning} can be generalized by replacing the identity $|qz| = |q||z|$ with $|q\Theta z|=q|z|$ (resp. $|q^{-1} z|=|q|^{-1}|z|$ with $|q^{-1}\Theta^{-1} z|=q^{-1}|z|$).

We note that $C_{\X}$ depends only on the region $\X$. 
If $|q|$ is sufficiently large compared to $\sum_{j=1}^{n}|q|^{n-j}|d_j|$, then we 
see $C_{\Omega}<|q|^n$. In this setting, we can choose $(\alpha,\beta)\in (0,1)^2$ satisfying the conditions of Theorem \ref{notwinning} because $C_{\Omega}|q|^{-n}<1$. 

In the proof above for $n=1$, we see that, in any step, we can replace $d$ by another digit $d'$ where $|d'|\leq |d|$.
Hence, we have the following corollary.

\begin{corollary}
    Let $0\leq t\in\R$. Then there exists $C_t>0$ satisfying the following: If $q\in \H$ with $|q|>1$ and $\alpha,\beta\in(0,1)$ such that $C_t \beta \le  \alpha\beta |q|= 1$, then the set 
    $\left\{z\in\mathcal{X}: d \text{ appears in } \mathbbm{d}_q(z) \text{ where } |d|\leq t\right\}$
    is $(\alpha,\beta)$-losing.
\end{corollary}

Moreover, in the proof of Theorem \ref{notwinning} for $n=1$, we want to make $C_d$ as small as possible. 
Hence, we want Bob to choose $\xi$ and $\rho$ so that $\overline{B(\xi, \rho)}$ is the biggest closed ball contained in $\mathcal{X}$. 
Likewise, we want to minimize $D$ and maximize $\rho$ and $|\xi|-2\rho$.

In Examples \ref{CDforrot} and \ref{CDforzeta}, we choose $\rho$ such that 
\[1+\frac{D}{\rho} =\frac{M+|d|}{|\xi|-2\rho}\]
in order to control the value of $C_d$ and provide a lower bound.

\begin{example}\label{CDforrot}
    In Example \ref{q-expansion}, we take $\{\eta_1, \eta_2, \eta_3, \eta_4\} =\{1, \i, \j, \k\}$.
    Let $\xi = (1+\i+\j+\k)/2 \in \mathcal{X}$.
    Note that $|\xi|=1$, $\sup_{z\in\mathcal{X}}|z-\xi|=1$ and $\sup_{z\in\mathcal{X}}|z|=2$.
    For $d\in \mathcal{D}$,
    we take
    \[\rho := \frac{\sqrt{|d|^2+6|d|+17}-|d|-3}{4}\leq \frac{\sqrt{17}-3}{4}<\frac{1}{2}.\]
    Then $|\xi| > 2\rho$ and $\overline{B(\xi,\rho)}\subseteq\mathcal{X}$.
    Also, $C_d = 1+1/\rho \geq 1+4/(\sqrt{17}-3)\approx 4.56$.
\end{example}

\begin{example}\label{CDforzeta}
    In Example \ref{zetaexpansion}, suppose $\mathrm{Re}(\zeta)\geq0$ and $\varepsilon\neq 1/2$.
    Then $M=\sup_{z\in\mathcal{X}}|z| = \sqrt{2}|1+\overline{\zeta}|\max\{1-\varepsilon,\varepsilon\}$.
    Take $\xi=(1/2-\varepsilon)(1+\overline{\zeta}+\eta+\overline{\zeta}\eta)$. Then $|\xi|=\sqrt{2}|1/2-\varepsilon||1+\overline{\zeta}|$.
    The largest ball that fits inside $\mathcal{X}$ is $B(\xi,1/2)$.
    Moreover, $D=\sup_{z\in\mathcal{X}}|z-\xi| = \sqrt{2}|1+\overline{\zeta}|/2$, which is minimized by our choice of $\xi$.
    For $d \in \mathcal{D}$, we take $\rho<1$ to be the positive solution to $1+D/\rho=(M+|d|)/(|\xi|-2\rho)$.
    Note that $\rho$ decreases as $|d|$ increases
    and
    $\rho\leq   \sqrt{13-3\sqrt{17}}|1+\overline{\zeta}|/4  \approx 0.7942 |1+\overline{\zeta}|.$
    Therefore, $C_d= 1+D/\rho \geq 4.5616$.
\end{example}

\begin{example}
    Let $\L=\mathbb{H}_H$ and $\mathcal{X} = \{a_1+a_2\i+a_3\j+a_4\k: 0\leq a_1,a_2,a_3<1, 0\leq a_4<1/2\}$.
    Take $\xi = (2+2\i+2\j+\k)/4$.
    Then $\sup_{z\in\mathcal{X}}|z|=\sqrt{13}/2$ and  $\sup_{z\in\mathcal{X}}|z-\xi|=\sqrt{13}/4$.
    Take $\rho=1/4$.
    Then $C_d =\max\{1+\sqrt{13}, (4|d|+2\sqrt{13})/(\sqrt{13}-2)\}$.
    If $d=0$, then $C_d=1+\sqrt{13}$. If $d\neq 0$, then $C_d = (4|d|+2\sqrt{13})/(\sqrt{13}-2) \ge 2(17+4\sqrt{13})/9 \approx 6.98$. 
\end{example}

\begin{example} 
    Consider the $\varepsilon$-symmetric $q$-expansion in Example \ref{symmetric}. We modify the proof of Theorem \ref{notwinning} slightly as the optimal choice for $\xi$ so that $B(\xi, \rho)$ is contained in $\mathcal{X}$ is zero and, in this case, $|\xi|=0$.
    Now, $M=\sup_{z\in\mathcal{X}}|z|=2\varepsilon$.
    Let $0<\tau<\varepsilon/2$ and $\xi=\tau(1+\i+\j+\k)\in \mathcal{X}$.
    Thus, $D=\sup_{z\in\mathcal{X}}|z-\xi| = 2(\varepsilon+\tau)$. 
    Take $\rho = \tau$. Then
    $\overline{B(\xi,\rho)}\subseteq \mathcal{X}$.
    Redefine 
    \[C_d = \max\{3+2\varepsilon/\tau, (2\varepsilon+|d|)/(2\tau)\}.\]
    If  $|d|\leq 6\tau+2\varepsilon$, then $C_d=3+2\varepsilon/\tau >7$.
    Otherwise, $C_d = (2\varepsilon+|d|)/(2\tau)$.
    
    Assume that $C_d/|q|\leq \alpha$.
    Let Bob start with the ball $B_0=\overline{B(\xi,\rho)} = \overline{B(\xi,\tau)}$.
    By calculations, $|z|\geq \tau$ for $z\in B_0$. 

    Using this configuration, Bob can use similar strategy as in the proof of Theorem \ref{notwinning} for $n=1$ so that $\mathcal{C}_q[d]$ is $(\alpha,\beta)$-losing.
\end{example}

Now, the interval $[C_\Omega/|q|^n,1)\cap (|q|^{-n},1)$ is the set of $\alpha$'s satisfying the hypothesis of Theorem \ref{notwinning}.
Note that in most cases,we can take $C_\Omega=M_n:= \frac{M+\left|\sum_{j=1}^n q^{n-j}d_j\right|}{|\xi|-2\rho}$.
For all sufficiently large $|q|$ depending on $\Omega$ and $\X$,
\[\frac{M_{n-1}/|q|^{n-1}}{M_n/|q|^n} = |q|\frac{M+\left|\sum_{j=1}^{n-1}q^{n-1-j}d_j\right|}{M+\left|\sum_{j=1}^n q^{n-j}d_j\right|}=|q|\frac{M+O(1)}{M+O(1)}\gg 1,\]
where $O$ is the Landau symbol.

This implies that, when $|q|$ is large, the set of $\alpha$'s for which $\mathcal{C}_q[\Omega]$ is $\alpha$-losing gets bigger as the length of $\Omega$ increases.
This agrees with our expectations since $\mathcal{C}_q[\Omega]$ gets smaller when $\Omega$ has more digits.

 For an admissible block $\Omega$ of length $|\Omega|$, we define
\[\widetilde{C_\Omega}:=\inf\{C>0: \mathcal{C}_q[\Omega]\text{ is } (\alpha,\beta)\text{-losing if } C\beta\leq \alpha\beta|q|^{|\Omega|}= 1\}.\]
We remark that $\widetilde{C_\Omega}$ complements the idea of the \textit{winning dimension} of $\mathcal{C}_q[\Omega]$.
Computing $\widetilde{C_\Omega}$ may not be trivial but $C_\Omega$ gives an upper bound for $\widetilde{C_\Omega}$. 
In the following example, we show that $C_\Omega$ given in Theorem \ref{notwinning} may be improved when $\X$ and $\L$ are fixed.

\begin{example}
    Let $q\in\H$ with $|q|>1$ and $\Omega=0$. 
    Consider the lattice $\L=\mathbb{H}_L$ with fundamental domain
    $\X = \{a+b\i+c\j+d\k \mid 0\leq a,b,c,d<1\}$.
    Note that $M=\sup_{z\in\X}|z|=2$.
    Let $\rho=2/5$ and $\xi = (1+\i+\j+\k)/2$.
    Then $D=\max_{z\in\X}|\xi-z|=1$.
    In the proof of Theorem \ref{notwinning}, 
    \[C_\Omega = \max\{1+D/\rho, M/(|\xi|-2\rho)\} = \max\{7/2, 10\} = 10.\]
    Hence, $\widetilde{C_\Omega}\leq 10$.
    In the following, we show that $\widetilde{C_\Omega}\leq 5$.
    Let $C\geq 5$.
    Suppose $C/|q|\leq \alpha$ and $\alpha\beta= 1/|q|$.
    Let Bob start with the ball $B_0=B(\xi,\rho)$.
    For $k\in\N$, let $A_k=\overline{B(a_k,\alpha^k\beta^{k-1}\rho)}$ and $B_k=\overline{B(a_k,\alpha^k\beta^{k}\rho)}$ be Alice's and Bob's $k$th ball, respectively. 
    Also, we denote by $a_{k}^{(m)}$ the $m$th digit of the $q$-expansion of $a_k$.
    Then Bob's strategy in the proof of Theorem \ref{notwinning} works, that is, Bob chooses 
    \[b_{k} = q^{-k}\xi+\sum_{j=1}^k a_j^{(j)}.\]

    This can be shown by induction. 
    We only show the first iteration as the induction step follows similarly as in the proof of Theorem \ref{notwinning}.
    Thus, we show that 
    \begin{itemize}
        \item $a_1^{(1)}\neq0$
        \item $b_{1}$ as defined in this strategy is a valid choice, i.e, $\overline{B(b_1,\alpha\beta\rho)}\subseteq A_1$.
        \item For any $z\in A_1$, the first digit of $z$ is $a_1^{(1)}\neq0$.
    \end{itemize}
    
    Connect $0$ and $a_1$ by a line segment and choose $z_0$ by moving from $a_1$ a distance of $\alpha\rho$ towards $0$.
    Then $z_0\in A_1\subseteq B_0$ and $|z_0| = |a_1|-\alpha\rho$.
    Suppose that $a_1^{(1)}=0$.  
    Then $qa_1\in\X$ and so, 
    \[|a_1|\leq \frac{2}{|q|}=\frac{2C}{C|q|}\leq \frac{2}{C}\alpha\leq \frac{2}{5}\alpha=\alpha\rho.\]
    This means $|z| = |a_1|-\alpha\rho\leq 0$, a contradiction.
    We have a contradiction.
    Thus, $a_1^{(1)}\neq0$. 

    Now, let $z\in B_1$.
    Then $|z-b_1|\leq \alpha\beta\rho$.
    We can write $a_1=q^{-1}a_1^{(1)}+q^{-1}t_1$ for some $t_1\in\X$ and $b_1 = q^{-1}\xi+q^{-1}a_1^{(1)}$.
    Note that $\max_{z\in\X}|\xi-z|=1$.
    Observe that
    \begin{align*}
        |z-a_1|&\leq |z-b_1|+|b_1-a_1| \leq \alpha\beta\rho+|q|^{-1}|\xi-t_1|\leq \alpha\beta\rho+|q|^{-1}\\
        &\leq \alpha\beta\rho+\frac{C}{C|q|}\leq \alpha\beta\rho+\frac{\alpha}{C} = \alpha\left(\beta\rho+\frac{1}{C}\right).
    \end{align*}
    Since $\frac{C\beta}{|q|}\leq\alpha\beta=\frac{1}{|q|}$, we have $\beta\leq C^{-1}\leq 1/5$.
    Thus, 
    \[\rho(1-\beta)=\frac{2(1-\beta)}{5}\geq \frac{8}{25}\geq \frac{1}{5}\geq \frac{1}{C}.\]
    This implies that
    \[|z-a_1|\leq \alpha\left(\beta\rho+\frac{1}{C}\right) \leq \alpha\rho.\] 
    So, $z\in A_1$ and $B_1\subseteq A_1$.
    Finally, we obtain that 
    \[|q(z-b_1)|\leq \alpha\beta\rho|q| = \rho\] and so,
    \[\rho \geq |q(z-b_1)| = |(qz-a_1^{(1)})-\xi|.\]
    Hence, $qz-a_1^{(1)}\in \overline{B(\xi,\rho)}\subseteq\X$.
    Thus, the first digit of $z\in B_1$ is $a_1^{(1)}\neq0$.
\end{example}

Given $d\in\mathcal{D}$, finding $(\alpha,\beta)$ such that $C_q[d]$ is $(\alpha,\beta)$-winning is relatively difficult since the cylinder sets are not so easy to compute. 
However, using Theorem \ref{dwinning}, we have the following result when $q\in\R$.
Let $K_q$ be the maximal length of zero blocks in the $q$-expansion 
(i.e., the real expansion on $[0,1)$ with radix $q$ described in Section 2) of $q-\lfloor q\rfloor$.
\begin{theorem}
    Let $1<q\in\R$, $\L=\mathbb{H}_L$ and $\mathcal{X}\cong [0,1)^4$.
    Suppose $K_q<\infty$.
    If $a_i \in \{0,1,2,\dots,\lfloor q\rfloor\}$ such that $a_i\leq d'$ where $d'$ is the minimal digit of $\lim_{\varepsilon\to0^+}\mathbbm{d}(1-\varepsilon)$ for $1\leq i\leq 4$ and $\alpha,\beta\in(0,1)$ such that $\beta>A_{2q}(\alpha)$ with $\log_q(\alpha\beta)\notin\Q$, then $\mathcal{C}_q[a_1+a_2\i+a_3\j+a_4\k]$ is $(\alpha,\beta)$-winning.
\end{theorem}

\begin{proof}
    Let $x_i$ be the center of Bob's $i$th ball $B_i$ and $y_i=a^{(1)}_i+a^{(2)}_i\i+a^{(3)}_i\j+a^{(4)}_i\k$ be the center of Alice's $i$th ball $A_i$.

    Since $\beta>A_{2q}(\alpha)$, similar to the proof of Theorem \ref{dwinning}, there exists $n,k\in\N$ such that for $j=1,2,3,4,$ Alice can force $a^{(j)}_{n+1}$ to be a center of some interval of $V_k(q;a_j)$. 
    In particular, $(K+2)/b^{k-1}<\rho(\alpha\beta)^n(1-\alpha)/2$ and the consecutive centers of the intervals of $V_k(q;a_j)$ are at most $q^{-(k-1)}$ units apart.
    Let $b_n^{(j)}$ be the $j$th component of $x_n$ and $a^{(j)}$ be the center of an interval of $V_k(q;a_j)$ that is closest to $b_n^{(j)}$.
    If the length of the interval of $V_k(q;a_j)$ centered at $a^{(j)}$ is $q^{-k}$, then Alice chooses $a_{n+1}^{(j)}=a^{(j)}$.
    Otherwise, Alice chooses $a_{n+1}^{(j)}$ to be the nearest center of an interval of $V_k(q;a_j)$ with length $q^{-k}$.
    In any case, $|a_{n+1}^{(j)}-b_n^{(j)}|\leq (K+2)q^{-(k-1)}<\rho(\alpha\beta)^n(1-\alpha)/2$.
    Thus,
    \begin{align*}
        |y_{n+1}-x_n|&=\sqrt{\sum_{j=1}^4|a_{n+1}^{(j)}-b_n^{(j)}|^2}< \sqrt{\sum_{j=1}^4\left(\frac{\rho(\alpha\beta)^n(1-\alpha)}{2}\right)^2} = \rho(\alpha\beta)^n(1-\alpha).
    \end{align*}
    
    We consider the case where Bob plays optimally.
    Now, Bob wants to move away from $y_k$ as far as possible since the $j$th component of $y_k$ has the digit $a_j$ in its $q$-expansion.
    In particular, he wants to concentrate his movement on one of the components, since moving away from a component is enough for him to win.
    He also uses the same direction in his moves to maximize the distance.
    
    WLOG, suppose that Bob chooses to move along the first component towards the positive direction.
    In this case, we have $x_k=y_k+\rho\alpha(\alpha\beta)^{k-1}(1-\beta)$.
    In response, Alice chooses
    \[y_{k+1}= x_k-\rho(\alpha\beta)^k(1-\alpha).\]
    Hence, the game reduces to a 1-dimensional Schmidt game since the other 3 components do not move. 
\end{proof}

\section{Acknowledgements}

We greatly appreciate the referee's diligence in reviewing our manuscript and for providing insightful suggestions that significantly improved this work.
The first and third authors acknowledge the University of the Philippines Diliman's Institute of Mathematics Computer Research Laboratory for funding support.
The second author's work was supported by the JSPS KAKENHI Grant Number 24K06641. 
The third author also acknowledges the Office of the Chancellor of the University of the Philippines Diliman, through the Office of the Vice Chancellor for Research and Development, for funding support through the Thesis and Dissertation Grant (242405 TND).

\bibliographystyle{siam}
\bibliography{references.bib}

\newpage

\section*{Corrections}

\medskip

The notion of Schmidt game is defined on a complete metric space $(X, \lambda)$, by using a subset $S\subseteq X$. 
In the original paper, the game was defined on a metric space $\X\subseteq \R^m$, by using $C[\Omega] \subseteq \X$. 
Throughout the paper, it is necessary to consider the Schmidt game on the topological closure $\overline{\X}$ of $\X$ to ensure the completeness of the whole space. However, it is possible to prove our results for the set $C[\Omega] \subseteq \X\subseteq\overline{\X}$ because the points $\overline{\X}\backslash \X$ can be avoided. In fact, we first consider theorems on $(\alpha,\beta,\rho)$-winning game. 
For example, in the same way as the proof of Theorem 2.6, Alice can choose a ball $A_{n+m+1}$ so that $A_{n+m+1} \subseteq V_k(b;d)\subseteq [0,1)$. Next, we consider $(\alpha,\beta)$-losing game. In the proof of Theorem 4.1, we choose sufficiently small $\rho$. Thus, in the same way as the proof of Theorem 4.1, Bob can choose $B_0$ so that $B_0\subset \X$. \par

In the proof of Theorem 3.9, we need to introduce one additional strategy. In fact, Alice uses the following inequality for the winning strategy: 
\[\rho(\alpha\beta)^n\alpha-\frac{2\rho(\alpha\beta)^{n+1}(1-\alpha)}{1-\alpha\beta}\leq \frac{1}{2r^k}.\]
Alice must avoid the points $\overline{C_\xi[0]}\backslash C_\xi[0]$. Alice chooses $A_{n+m+1}$ in the same way as the proof of Theorem 3.9 and selects $A_{n+m+2}$ so that $A_{n+m+2}$ is included in the topological interior of $B_{n+m+1}$. 

\end{document}